\documentclass[12pt]{article}
\usepackage{amsmath, amscd, amssymb,latexsym, epsfig, color}
\input xy
\xyoption{all}

\newtheorem{theorem}{Theorem}[section]
\newtheorem{proposition}[theorem]{Proposition}
\newtheorem{question}[theorem]{Question}
\newtheorem{corollary}[theorem]{Corollary}
\newtheorem{definition}[theorem]{Definition}
\newtheorem{conjecture}[theorem]{Conjecture}
\newtheorem{remark}[theorem]{Remark}
\newtheorem{lemma}[theorem]{Lemma}
\newtheorem{example}[theorem]{Example}

\newtheorem{problem}[theorem]{Problem}

 \def\Z{\mathbb{Z}}

\def\C{\mathbb{C}}

\def\R{\mathbb{R}}

\newcommand{\lk}{\mbox{\upshape lk}\,}

\newcommand{\field}{{\mathbb{F}}}

\newcommand{\Star}{\mbox{\upshape st}\,}

\newcommand{\soc}{\mbox{\upshape Soc}\,}

\title{Thirty-five years and counting}
\author{Ed Swartz
\thanks{Research partially supported by NSF grant DMS-1200478}\\
\small Department of Mathematics, \\[-0.8ex]
\small Cornell University, Ithaca NY, 14853-4201, USA, \\[-0.8ex]
\small \texttt{ebs22@cornell.edu } }

\begin{document}
\maketitle

\begin{abstract}

It has been  35 years since Stanley  proved that $f$-vectors of boundaries of simplicial polytopes satisfy McMullen's conjectured $g$-conditions \cite{St80}.  Since then one of the outstanding questions in the realm of face enumeration is whether or not Stanley's proof  could be extended to larger classes of spheres.  Here we hope to give an overview of   various attempts to accomplish this and why we feel this is so important.  In particular, we will see a strong connection  to $f$-vectors of  manifolds and pseudomanifolds.  Along the way we have included several previously unpublished results involving how the $g$-conjecture  relates to bistellar moves and small $g_2,$ the topology and combinatorics of stacked manifolds introduced independently by  Bagchi and Datta \cite{BD2} and Murai and Nevo \cite{MN}, and counterexamples to over optimistic generalizations of the $g$-theorem. 
\end{abstract}

\section{Introduction}

In 1971 McMullen proposed a complete characterization of the $f$-vectors of simplicial polytopes which eventually become known as the $g$-conjecture \cite{McM71}.  When describing this conjecture in their 1971 book, {\it Convex Polytopes and the Upper Bound Conjecture},  McMullen and Shephard wrote, ``Even more intriguing, if rather less plausible, is the following conjecture proposed in [14]." (The reference [14] was a preprint for \cite{McM71}.) Less than ten years later the $g$-conjecture was a theorem.  Billera and Lee proved the sufficiency of the conditions \cite{BilleraLee} and Stanley proved their necessity \cite{St80}.  

At the end of the paper in which  McMullen proposed his $g$-conjecture he offered an opinion as to its applicability  to more general triangulated spheres  \cite{McM71},  ``However, there are fundamental differences between triangulated $(d-1)$-spheres and boundaries of simplicial $d$-polytopes.....We should therefore, perhaps, be wary of extending the conjecture to triangulated spheres." Thirty-five years after a positive resolution of McMullen's original $g$-conjecture for polytopes, the fate of the same $g$-conjecture for broader classes of spheres remains open.  Here we will survey several  attempts to extend the class of spaces to which the $g$-theorem might apply.  Enough time has passed since Stanley's original proof that many researchers, including this author, refer to  the possibility of proving the $g$-theorem for other spheres as a $g$-conjecture and McMullen's original proposal as the $g$-theorem.    Section \ref{manifold section} examines  results over the last ten years which point to a close relationship between the new $g$-conjecture(s) and $f$-vectors of  manifolds and  pseudomanifolds without boundary.   Sections \ref{PL-spheres} and \ref{small g2} consist of previously unpublished results concerning $g$-conjectures for PL-spheres from the point of view of bistellar moves, and results for homology spheres with few vertices or few edges.  The last two sections look at the possibility of extending $g$-conjectures to $l$-Cohen Macaulay complexes and how $g$-conjectures relate to balls and other manifolds with boundary, especially the recently introduced stacked manifolds of Bagchi and Datta \cite{BD2} and Murai and Nevo \cite{MN}.

Throughout $\Delta$ will be a connected $(d-1)$-dimensional pure simplicial complex.  All of our complexes will be finite.  We denote the geometric realization of $\Delta$ by $|\Delta|.$  The {\bf $f$-vector} of $\Delta$ is $( f_0, \dots, f_{d-1}),$ where $f_i$ is the number of $i$-dimensional faces in $\Delta.$ The vertices of $\Delta$ are $[1,f_0]=\{1, \dots, f_0\}.$ Usually we include the empty face in the $f$-vector data and set $f_{-1} = 1.$ The  face polynomial of $\Delta$ is 

$$f_\Delta(x) = f_{-1} x^d + f_0 x^{d-1} + \dots + f_{d-2} x + f_{d-1}.$$

The {\bf $h$-vector} of $\Delta$ is $(h_0, \dots, h_d)$ and is defined so that the corresponding $h$-polynomial, $h_\Delta(x) = h_0 x^d + h_1 x^{d-1} + \dots +h_{d-1} x +  h_d,$ satisfies $h_\Delta(x+1) = f_\Delta(x).$  Equivalently,

\begin{equation} \label{f by h}
  h_i = \sum^i_{j=0} (-1)^{i-j} \binom{d-j}{d-i} f_{j-1}.
\end{equation} 
 
Each $f_{i-1}$ is a   nonnegative linear combination of  $h_0, \dots, h_i.$  Specifically,
\begin{equation} \label{h by f}
  f_{i-1} = \sum^i_{j=0} \binom{d-j}{d-i} h_j.
\end{equation}

As the above equations demonstrate, $(f_{-1}, f_0, \dots, f_{i-1})$ contains exactly the same information as $(h_0, \dots, h_i).$  For instance, $h_d= (-1)^{d-1}\tilde{\chi}(|\Delta|)$ and $f_{d-1} = h_0+ \dots + h_d.$  There are many reasons to use the $h$-vector encoding of the $f$-vector.  In addition to the connection to the algebra we will see later, many formulas are more easily understood in the $h$-vector format.    For instance, the linear relations among the $f_i$ when $|\Delta|$ is a manifold are given by a simple formula due to Klee.   Klee's equations hold in the more general setting of semi-Eulerian complexes.  A pure $(d-1)$-dimensional simplicial complex $\Delta$ is {\bf semi-Eulerian} if for every nonempty face $\sigma$ of $\Delta$ the Euler characteristic of the link of $\sigma,~\lk \sigma,$ equals the Euler characteristic of $S^{d-|\sigma|-1},$ the $(d-|\sigma|-1)$-dimensional sphere.    As usual the link of $\sigma$ is $\lk \sigma = \{ \tau \in \Delta: \tau \cup \sigma \in \Delta, \tau \cap \sigma = \emptyset.\}$

 \begin{theorem}(Klee) \label{Klee} \cite{Klee}  
     Suppose $|\Delta|$ is semi-Eulerian.    Then
        \begin{equation} \label{klee}
        h_{d-i} - h_i =  (-1)^i \binom{d}{i} (\chi (\Delta) - \chi (S^{d-1})).
       \end{equation}
    \end{theorem}
    
    The prototypical example of a semi-Eulerian complex is a triangulation of an $\field$-homology manifold.  Except where otherwise indicated, $\field$ will be an infinite field of arbitrary characteristic.  Homology and the corresponding Betti numbers will always be reduced homology with $\field$-coefficients.  A complex $\Delta$ is an $\field$-{\bf homology manifold} if the $\field$-homology of the link of every nonempty face $\sigma$ is isomorphic to the $\field$-homology of $S^{d-|\sigma|-1}.$  We say that $\Delta$ is an $\field$-{\bf homology manifold with boundary} if the link of every nonempty face $\sigma$ of $\Delta$ has the $\field$-homology of either $S^{d-|\sigma|-1}$ or $B^{d-|\sigma|-1}$ (the $(d-|\sigma|-1)$-dimensional ball) and the union of the empty set and the collection of faces which satisfy the latter condition, called the {\bf boundary faces}, form a $(d-2)$-dimensional $\field$-homology manifold.

From  (\ref{klee})  we see that if $|\Delta|$ is an $\field$-homology manifold and we know its Euler characteristic, then  the first half of the $h$-vector contains all of the $h$-vector (and hence $f$-vector) information.  So it is possible to encode the face numbers in the $g$-vector which is defined as
$(1, g_1, \dots, g_{\lfloor d/2 \rfloor})$ where $g_i = h_i - h_{i-1}.$ The $g$-theorem relates the $g$-vectors of boundaries of simplicial polytopes to M-vectors.  In order to define M-vectors we first introduce Macaulay pseudo-powers.  Given $a$ and $i$ positive integers there is a unique way to  write
$$a = \binom{a_i}{i} + \binom{a_{i-1}}{i-1} + \dots + \binom{a_j}{j},$$
with $a_i > a_{i-1} > \dots > a_j \ge j \ge 1.$

Define 
$$a^{<i>} = \binom{a_i+1}{i+1} + \binom{a_{i-1}+1}{i} + \dots + \binom{a_j+1}{j+1}.$$
\noindent In the literature,  $a^{<i>}$ is frequently called a (Macaulay) pseudo-power.  For convenience we define $0^{<i>} = 0$ for all $i.$  

\begin{theorem} (Macaulay) \cite{Mac} \label{Macaulay}
  Let $(h_0, \dots, h_d)$ be a sequence of nonnegative integers. Then        $(h_0, \dots, h_d)$ is the Hilbert function of a homogeneous quotient of a polynomial ring if and only if 
       $h_0=1$ and $h_{i+1} \le h^{<i>}_i$ for all $1 \le i \le d-1.$
    \end{theorem}

\noindent  Any sequence $(h_0, \dots, h_d)$ which satisfies the above inequalities is called an  {\bf M-vector}.   

\begin{theorem} ($g$-theorem) \label{g-thm} \cite{St80}, \cite{McM93}, \cite{McM96}, \cite{BilleraLee}
Let $\Delta$ be the boundary of a simplicial $d$-polytope.  Then its $g$-vector is an M-vector.  Conversely, suppose $(1, g_1, \dots, g_{\lfloor d/2 \rfloor})$ is an M-vector. Then there exists a simplicial $d$-polytope such that the $g$-vector of its boundary equals $(1, g_1, \dots, g_{\lfloor d/2 \rfloor}).$
\end{theorem}

Billera and Lee  established that for every M-vector there was a simplicial polytope with the given $g$-vector \cite{BilleraLee}.  Stanley's proof of necessity used complicated and difficult algebraic geometry \cite{St80}.  Later, McMullen  gave another proof using complicated and difficult convex geometry \cite{McM93}, \cite{McM96}.   The idea behind both proofs was to find a ring whose Hilbert function equals the $g$-vector of the polytope.   

The {\bf face ring} or {\bf Stanley-Reisner} ring of $\Delta$ is the polynomial ring $R=\field [x_1, \dots, x_{f_0}]$ modulo the {\bf face ideal}, $I_\Delta.$  The face ideal is generated by the monomials corresponding to the nonfaces of $\Delta.$  Specifically, 

$$I_\Delta = <x_{i_1} \cdot  \dots \cdot  x_{i_m}: \{i_1, \dots, i_m\} \notin \Delta.>$$

We use $\field[\Delta]$ for  the face ring of $\Delta$.  Let $\{\theta_1, \dots, \theta_d\}$ be a set of $d$ one-forms in the polynomial ring $R$ and let $\Theta$ be the ideal they generate.  Then $\{\theta_1, \dots, \theta_d\}$ is called a {\bf linear system of parameters}, or l.s.o.p., for $\field[\Delta]$ if $\field[\Delta]/\Theta$ is finite-dimensional as a vector space over $\field.$ Even though  the quotient $\field[\Delta]/\Theta$  depends on the choice of l.s.o.p., we denote it by $\field(\Delta).$ As long as $\field$ is infinite,  generic choices of $\{\theta_1, \dots, \theta_d\}$ are a l.s.o.p. for $\field[\Delta].$ Since $\Theta$ and $I_\Delta$ are homogeneous ideals, $\field(\Delta)$ is a graded ring.  We write the component of degree $i$ of $\field(\Delta)$ as $\field(\Delta)_i.$  The reason for introducing these ideas is  Schenzel's formula which relates the topology of $|\Delta|$, its $h$-vector, and the Hilbert function of $\field(\Delta)$ whenever $|\Delta|$ is an $\field$-homology manifold (with or without boundary).

\begin{theorem}  (Schenzel's formula) \cite{Sch}
Suppose $|\Delta|$ is an  $\field$-homology manifold.  Let $\{\theta_1, \dots, \theta_d\}$ be a l.s.o.p. for $\field[\Delta].$  Then 
$$\dim_\field \field(\Delta)_i = h_i(\Delta) + \binom{d}{i} \sum^{i-2}_{j=0} (-1)^{j-i} \beta_j.$$
 \end{theorem}

The above formula is a special case of Schenzel's work on Buchsbaum complexes \cite{Sch}.  Schenzel's formula is a generalization of the original work of Reisner \cite{Rei} and Stanley \cite{St77} on Cohen-Macaulay complexes, of which we will say more in Section \ref{k-cm}.

For $\field$-homology spheres all of the relevant Betti numbers are zero, hence $\dim_\field \field(\Delta)_i = h_i(\Delta).$ The key fact in the proof of the $g$-theorem is the existence of Lefschetz elements for $\field(\Delta).$  A one-form $\omega$ in $R$ is a {\bf Lefschetz} element for $\field(\Delta)$ if multiplication 
$$\cdot \omega^{d-2i}: \field(\Delta)_i \to \field(\Delta)_{d-i}$$
is an isomorphism for all $i \le  d/2.$

\begin{theorem} \cite{St80}, \cite{McM93}, \cite{McM96}
If $\Delta$ is the boundary of a simplicial $d$-polytope, then there exists a l.s.o.p. $\{\theta_1, \dots, \theta_d\}$ and a Lefschetz element $\omega$ for $\C(\Delta).$ 
\end{theorem}

In Stanley's toric variety proof the choice of l.s.o.p.  depends on a geometric embedding of the polytope with rational vertices. In McMullen's proof both the l.s.o.p. and the Lefschetz element depend on the polytope data.  In either case,  once there exists a l.s.o.p. and a Lefschetz element, generic choices for either work.  Here and throughout the rest of the paper generic will always mean `within a nonempty open Zariski set of $\field^N$', where $N$ depends on the setting.  For instance, for choices of a l.s.o.p for $\Delta,~N$ is $f_0d.$   In general we will say that $\field[\Delta]$ has Lefschetz elements whenever for generic choices of l.s.o.p. and one-form $\omega$ the latter is a Lefchetz element for $\field(\Delta).$ To finish the proof of the necessity of the $g$-conditions one observes that in order for $\cdot \omega: \field(\Delta)_i \to \field(\Delta)_{d-i}$ to be an isomorphism, multiplication by $\omega,$
$$\cdot \omega: \field(\Delta)_i \to \field(\Delta)_{i+1}$$ must be an injection whenever $i < d/2.$  Hence the Hilbert function of $\field(\Delta)/(\omega)$ equals the $g$-vector of $\Delta$ for $i \le d/2.$

There are potentially many ways to  extend the $g$-theorem to larger classes of spheres.  For sphere-like spaces the most optimistic version is to hope that $\field[\Delta]$ has Lefschetz elements whenever $|\Delta|$ is a rational homology sphere.  There are, however, many classes of spheres between boundaries of simplicial polytopes and rational homology spheres.  These include shellable, locally collapsible, PL, integral homology, and simplicial, to name just a few. 

The full strength of Lefschetz elements is not necessary to establish that the $g$-vector is an M-vector.  A common weaker requirement involves the existence of weak Lefschetz elements.  A one-form $\omega \in R_1$ is a {\bf weak Lefschetz} element for $\field[\Delta]$ if multiplication $\cdot \omega: \field(\Delta)_i \to \field(\Delta)_{i+1}$ is either an injection or surjection for all $0 \le i \le d-1.$  When $\Delta$ is an $\field$-homology sphere, the Gorenstein property of $\field[\Delta]$ implies that it is sufficient to check that multiplication is surjective for $i=d/2$ when $d$ is even and $i=(d-1)/2$ when $d$ is odd.    As is the case for Lefschetz elements, the existence of one such l.s.o.p. and one-form implies that  generic choices of both work. Whether or not it is possible for $\field[\Delta]$ to have weak Lefschetz elements, but not Lefschetz elements is, as of now, unclear.  See \cite[Section 4]{BN} for a discussion.

 In addition, it might be possible to prove that $g$-vectors are M-vectors, or at least nonnegative,  without resorting to the face ring.  In his 1970 paper \cite{Wal}, before the connection between $f$-vectors and face rings was known, Walkup proved that arbitrary three-manifolds without boundary satisfy $g_2 \ge 0$ and that $g_2$ is also nonnegative when $|\Delta|$ is a four-manifold without boundary whose  Euler characteristic is at most two.   At present it is unknown whether or not $\field(\Delta)$ has Lefschetz elements when  $|\Delta|$ is a three sphere.  When $|\Delta|$ is  three-manifold with nontrivial $\beta_1,~\field(\Delta)$ never has Lefschetz elements.  However, it does have weak Lefschetz elements.  It remains an open problem whether or not triangulations of four-spheres have weak Lefschetz elements.   A four-dimensional manifold with $\beta_1 > 0$ and $\beta_1 +\chi \ge 2$  does not even have weak Lefschetz elements. In Section \ref{k-cm} we will meet a class of complexes whose $h$-vectors are known to be unimodal, even log-concave, but do not have weak Lefschetz elements.  We will also discuss a new approach to the nonnegativity of the $g$-vector which does not a priori involve the face ring.   See the discussion after Corollary \ref{beta1}.

One interesting class of PL-spheres for which it is known that there exits Lefschetz elements for $\field[\Delta]$ are  strongly edge decomposable (s.e.d.) spheres.   Roughly speaking, a simplicial sphere is s.e.d. if is possible to repeatedly contract edges until you reach the boundary of the simplex.  Introduced by Nevo \cite{Nev2}, where he demonstrated that the $g$-vector is nonnegative, strongly edge decomposable spheres are PL-spheres.  The existence of Lefschetz elements for $\field[\Delta]$ in characteristic zero when $\Delta$ is an s.e.d. sphere   was proved by Babson and Nevo \cite{BN}, and in arbitrary characteristic by Murai \cite{Mur}.  How strongly edge decomposable spheres compare to other classes of spheres is not yet understood.  Novikov's proof that in dimensions five and above  there is no algorithm to recognize PL-spheres   implies that s.e.d. spheres cannot possibly contain all PL-spheres in dimensions five and above.  

For a different approach to establishing the existence of weak Lefschetz elements for face rings of PL-spheres see Remark \ref{KS} below.  

Another collection of spheres for which the conclusion of the $g$-theorem is known to hold are several classes related to barycentric subdivision.  A poset $P$ is Gorenstein* if its order complex $|P|$  is a rational homology sphere.  The $g$-vector of the order complex of a Gorenstein* poset $P$ can be written as a nonnegative linear combination of the coefficients of the ${\mathbf{cd}}$-index of  $P$.  So Karu's proof that these latter coefficients are nonnegative \cite{Kar} implies that the $g$-vector of order complexes of Gorenstein* complexes is also nonnegative.   The prototypical example of such a $P$ is the face poset of a rational homology sphere $\Delta$.  In this case $|P|$ is the first barycentric subdivision of $\Delta$ and  Kubitzke and Nevo showed that the $g$-vector of $|P|$ is also an M-vector \cite{KN}.  The Kubitzke-Nevo result is a special case of their main result which concerns barycentric subdivisions of Cohen-Macaulay complexes. Brenti and Welker had previously shown that the $h$-vector barycentric subdivision of any Cohen-Macaulay complex was log concave by proving the $h$-polynomial had only real zeroes \cite{BW}. Instead of proving that $\field[\Delta]$ has weak Lefschetz elements, Kubitzke and Nevo prove that shellable complexes with the same $f$-vector as $\Delta$ have weak Lefschetz elements. These results were  later extended by Murai and Yanagawa to order complexes of posets whose downsets have Lefschetz elements \cite{MY}.

  So far the best result toward establishing the existence of weak Lefschetz elements for rational homology spheres is the following, sometimes called the rigidity inequality.    A {\bf normal pseudomanifold} is a pure simplicial complex such that the links of all faces of codimension-two or greater are connected, and whose codimension-one faces are in exactly two facets.   Any triangulation of an $\field$-homology manifold will be a normal pseudomanifold.  

\begin{theorem}  \label{rigidity}(Rigidity inequality) \cite{Ka87}, \cite{Lee}
Let $\Delta$ be a normal pseudomanifold and $d \ge 4.$ If $\field$ is an infinite field, then for generic choices of l.s.o.p. $\Theta$ and  $\omega \in R_1,$
\begin{enumerate}
  \item
   $\dim_\field \field(\Delta)_1 = h_1(\Delta).$
  \item \label{rigidity h2}
    $\dim_\field \field(\Delta)_2 = h_2(\Delta).$
  \item \label{rigidity h3}
   $\dim_\field \field(\Delta)_3 \ge h_3(\Delta).$
  \item
   $\cdot \omega: \field(\Delta)_1 \to \field(\Delta)_2$ is injective.
\end{enumerate}

\end{theorem}

\begin{proof}
In characteristic zero this follows easily from the work of Kalai \cite{Ka87} and Lee \cite{Lee}.  For nonzero characteristic see the discussion and references in \cite[Section 5]{NovSw3}.
\end{proof}

 Since any two-dimensional sphere is combinatorially the boundary of a $3$-polytope, the $g$-theorem holds for all such spheres.  The rigidity inequality is enough to show that the $g$-theorem also holds for all spheres of dimension three or four.  

\begin{corollary}
 If $\Delta$ is a three or four-dimensional $\field$-homology sphere, then  $(1, g_1, g_2)$ is an M-vector.
\end{corollary}

The content of the corollary is that $g_2 \ge 0.$  The inequality $g_2 \le g^{<1>}_1$ is just a different way of saying that $f_0 \le {f_1 \choose 2}.$ Another application of the rigidity inequality is the following upper bound.

\begin{corollary} \label{upper bound}
Let $\Delta$ be a normal pseudomanifold with $d \ge 4.$  Then, $g_3 \le g^{<2>}_2.$  In particular, if $g_3 \ge 0,$ then $(1, g_1, g_2, g_3)$ is an M-vector.  Furthermore, if $g_3=g^{<2>}_2,$ then for generic $\omega$ multiplication $\cdot \omega: \field(\Delta)_2 \to \field(\Delta)_3$ is injective.  
\end{corollary}

\begin{proof}
Both conclusions follow by considering $\field(\Delta)/(\omega)$, where multiplication $\cdot \omega: \field(\Delta)_1 \to \field(\Delta)_2$ is injective, and then combining Theorem \ref{Macaulay} with the elementary observation that $\dim_\field (\field(\Delta)/(\omega))_3 = g_3 + \ker \cdot \omega.$   \end{proof}

Theorem \ref{rigidity} has been extended to doubly Cohen-Macaulay spaces.  See Theorem \ref{2-cm} below.

\section{Manifolds and pseudomanifolds} \label{manifold section}

In \cite{N98} Kalai conjectured a far reaching generalization of the $g$-conjecture to manifolds.  Here we explain this conjecture and its relationship to the $g$-conjectures for spheres. We also examine upper and lower bounds for $g$-vectors of manifolds and pseudomanifolds without boundary that are implied by $g$-conjectures.   There are several categories of triangulated manifolds one might consider.  Examples include PL, topological and homology manifolds.  We will state our results for homology manifolds as they are the most general, but essentially identical results hold for any category which is defined in terms of the links of the complex.

  Suppose $|\Delta|$ is a $(d-1)$-dimensional $\field$-homology manifold (without boundary).  In contrast to spheres, there is frequently no chance for $\cdot \omega:\field(\Delta)_i \to \field(\Delta)_{i+1}$ to be an injection.  The {\bf socle} of $\field(\Delta)$ is the ideal $\{s \in \field(\Delta): \omega \cdot s = 0 ~\forall~ \omega \in R_1\}.$  We denote the socle of $\field(\Delta)$ by $\soc(\field(\Delta)).$ The socle of $\field(\Delta)$ is graded and if $\soc(\field(\Delta))_i \neq 0,$ then there is obviously no hope of finding one-forms $\omega$ which give injections from degree $i$ to $i+1.$
  
  \begin{theorem}  \cite{NovSw} \label{socle dim}
  If $\Delta$ is an $\field$-homology manifold (with or without boundary), then $$\dim_\field \soc(\field(\Delta))_i \ge {d \choose i} \beta_{i-1}(\Delta).$$
  \end{theorem}
  
  To overcome this difficulty Kalai introduced $h''$ vectors which take into account the socle.  Let $h'_i(\Delta) =  h_i(\Delta) + \binom{d}{i} \sum^{i-2}_{j=0} (-1)^{j-i} \beta_j.$  This is the right-hand side of Schenzel's formula.  Now define $h''_i(\Delta) = h'_i(\Delta) - {d \choose i} \beta_{i-1}$ for $0 \le i<d$ and $h''_d=h'_d.$  
  
  \begin{theorem} \cite{N98}
  If $\Delta$ is an orientable $(d-1)$-dimensional $\field$-homology manifold (without boundary), then $h''_i = h''_{d-i}.$
  \end{theorem}
  
  \begin{conjecture} (Kalai's manifold $g$-conjecture) \cite{N98}
  If $\Delta$ is an orientable $(d-1)$-dimensional $\field$-homology manifold (without boundary), then $(h''_0, h''_1 - h''_0, \dots, h''_{\lfloor d/2 \rfloor} - h''_{\lfloor d/2 \rfloor-1})$ is an M-vector.
  \end{conjecture} 
  
  For homology spheres this is just the $g$-conjecture.  Amazingly, this conjecture is no stronger than the conjectured existence of weak Lefschetz elements for spheres.  
  
 \begin{theorem} \label{Kalai_conj} \cite{NovSw2}
   Let $\Delta$ be a $(d-1)$-dimensional $\field$-homology manifold.  Suppose that for all but possibly $d+1$ vertices  $v,~\field[\lk v ]$ has  weak Lefschetz elements.  Then Kalai's manifold conjecture holds for $\Delta.$
 \end{theorem} 
  
  The proof of this theorem is based on the following facts.  One, if $J=\oplus^{d-1}_{i=0} \soc(\field(\Delta))_i$ is the socle of $\field(\Delta)$ except for degree $d,$ then $\dim_\field (\field(\Delta)/J)_i= h''_i.$  Two, $\field(\Delta)/J$  is a Gorenstein ring \cite{NovSw2}.  This means that finding $\omega$ such that $\cdot \omega: \field(\Delta)_i \to \field(\Delta)_{i+1}$ is injective is the same as finding $\omega$ such that multiplication by $\omega$ from $\field(\Delta)_{d-i-1} \to \field(\Delta)_{d-i}$ is surjective.  Three, the following very general  lemma which bounds the size of the cokernel of multiplication by generic one-forms. 
  
  \begin{lemma} (cokernel lemma) \label{cokernel}
  Let $\Delta$ be a pure $(d-1)$-dimensional complex.  Let $V$ be the subset of vertices $v$ such that there do not  exist generic one-forms $\omega_v$ and l.s.o.p $\Theta_v$ such that multiplication $\cdot \omega_v: \field(\lk v)_{i-1} \to \field(\lk v )_i$ is surjective.  Then for generic $\omega$ the dimension of the cokernel of  multiplication $\cdot \omega: \field(\Delta)_i \to \field(\Delta)_{i+1} $ is at most the number of degree $i$ monomials in a polynomial ring with $|V|-d-1$ variables. 
   \end{lemma}
  
\begin{proof}
Let $V$ be the subset of vertices $v$ for which there do not exist generic one-forms $\omega$ such that multiplication $\cdot \omega: \field(\lk v)_{i-1} \to \field(\lk v)_i$ is surjective and let $\Delta_V$ be the induced subcomplex on this set of vertices.  For generic $\omega$ consider the commutative diagram
\begin{equation} \label{surjective commutative diagram}
\begin{array}{ccccccccc}
0&\to&I_{V_i}&\to&\field(\Delta)_i&\to&\field(\Delta_V)_i&\to&0\\
 & &\cdot \omega \uparrow& &\cdot \omega \uparrow& &\cdot \omega \uparrow& & \\
 0&\to&I_{V_{i-1}}&\to&\field(\Delta)_{i-1}&\to&\field(\Delta_V)_{i-1}&\to&0.
\end{array}
\end{equation}
The ideal $I_V$ is generated by the monomials $x_j,$ where there do exist generic one-forms whose multiplication induces surjectivity from degree $i-1$ to degree $i.$  Using the same reasoning as in the proof of \cite[Theorem 4.26]{Sw} there are natural surjections from $\field(\lk {v_j})$ to the ideal $(x_j) \subseteq \field(\Delta).$  Hence the l.h.s. arrow is a surjection and the snake lemma finishes the proof. 
\end{proof}

Kalai's manifold $g$-conjecture is a remarkable generalization of the $g$-conjecture.  However, in practice, sharp upper and lower bounds have been obtained by using the cokernel lemma to show  $\cdot \omega: \field(\Delta)_{d-i} \to \field(\Delta)_{d-i+1}$ is surjective when $i \ge d/2.$  This was the idea behind the following.

\begin{theorem} \cite{Sw}, \cite{NovSw}  
Let $\Delta$ be an  $\field$-homology manifold (with or without boundary) and $d \ge 4.$ Then for generic $\Theta$ and $\omega$ 
\begin{equation} \label{manifold rigidity}
\cdot \omega: \field(\Delta)_{d-2} \to \field(\Delta)_{d-1}
\end{equation}
is a surjection.  In particular, $h''_{d-2} \ge h'_{d-1}.$  Furthermore, if $\partial \Delta = \emptyset,\ h''_{d-2} = h'_{d-1}$ and $d \ge 5$,  then all the links of $\Delta$ are stacked spheres and $|\Delta|$ is either $S^{d-1}$ or a connected sum of $S^{d-2}$-bundles over $S^1.$  
\end{theorem} 

\noindent For the definition of stacked sphere see Section \ref{small g2}.

\begin{proof}
The surjectivity in (\ref{manifold rigidity}) is \cite[Corollary 4.29]{Sw}.  The inequality follows immediately from Theorem \ref{socle dim}.  The last statement  when $|\Delta|$ is orientable is \cite[Theorem 5.2]{NovSw}.  The only place orientability is used  in that proof is an appeal to \cite[Proposition 4.24]{Sw} to show that there are degree one isomorphisms from $\field(\lk j)$ to  the ideal generated by  $x_j$ in $\field(\Delta).$  While this is no longer true when $|\Delta|$ is not orientable since $\field(\lk j)_{d-1} \simeq \field$ and $(x_j)_d = 0,$ the argument in the proof of \cite[Proposition 4.24]{Sw} is still valid for lower degrees.  So the proof of \cite[Theorem 5,2]{NovSw}  also works for nonorientable $\Delta$.  
\end{proof}

\begin{corollary}  \label{beta1}
Suppose $\Delta$ is an $\field$-homology manifold  and $d \ge 4.$  
\begin{enumerate}
  \item \cite{NovSw} \label{orientable beta1}
  If $|\Delta|$ is orientable, then $g_2 \ge {d+1 \choose 2} \beta_1.$ If $d \ge 5$ and $g_2 = {d+1 \choose 2} \beta_1,$ then every vertex link is a stacked sphere and $|\Delta|$ is either $S^{d-1}$ or a  connected sum of copies of the orientable $S^{d-2}$-bundle over $S^1.$  
  \item If $|\Delta|$ is not orientable, then $g_2 \ge {d+1 \choose 2} (\beta_{d-2} + 1).$  If $d \ge 5$ and $g_2 = {d+1 \choose 2} (\beta_{d-2}+1),$ then every vertex link is a stacked sphere and $\Delta$ is a connected sum of copies of the nonorientable $S^{d-2}$-bundle over $S^1.$  
\end{enumerate}
\end{corollary}

\begin{remark}
The first item  is \cite[Theorem 5.2]{NovSw}.  The nonorientable inequality was known to the authors of \cite{NovSw} who eventually published a more general form for pseudomanifolds \cite[Theorem 4.9]{NovSw4}.  In this setting it follows easily from the previous theorem, the definitions of $h'$ and $h''$, Klee's equations, the reduced Euler characteristic as the alternating sum of the reduced Betti nmbers and the fact that $\beta_{d-1}=0$ when $\Delta$ is not orientable.   When $d=4$ the two statements are the same since $\beta_1 = \beta_2 +1$ for nonorientable homology three-manifolds. 
\end{remark}

The inequality in Corollary \ref{beta1}, (\ref{orientable beta1}.)  was originally conjectured by Kalai \cite[Conjecture 14.1]{Ka87}. The three-dimensional case was recently proven by Bagchi \cite{Bag} using methods that are rooted in PL-Morse theory.  While there is no apparent route to the upper bounds implicit in the $g$-conjecture, there is also no apparent reason to believe that this alternative approach could not produce a proof of $g_i \ge 0$ for $\field$-homology spheres.  Of course, by Corollary \ref{upper bound}, the upper bound on $g_3$ comes for free.   As evidence of the potential of the Morse-theoretic technique we point out that Bagchi also extended the equality case of Corollary \ref{beta1} to $d=4$ \cite{Bag}.  

What other upper and lower bounds are implied when the links of the vertices of an $\field$-homology manifold satisfy the $g$-conjecture?  To answer this Murai and Nevo introduced a variant of the $g$-vector.       For $0 \le i \le \lfloor d/2 \rfloor$ define $\overline{g}_i$ to be $h''_{d-i} - h'_{d-i+1}$, and $\hat{g}_i =   g_i +(-1)^{i+1} {d+1 \choose i} \sum^i_{j=1} (-1)^j \beta_{j-1}.$

   \begin{proposition}  \label{gtilde formula} Let $\Delta$ be an $\field$-homology manifold and $0 \le i \le d/2.$  
     \begin{equation} \label{tilde any} \overline{g}_i =  g_i + (-1)^{i+1} {d+1 \choose i}\left[ 1+ \sum^i_{j=1} (-1)^j \beta_{d-j} \right].
     \end{equation}
     \item  \cite{MN}
   If $|\Delta|$ is  $\field$-orientable, then 
   \begin{equation}\label{gtilde orientable}
   \hat{g}_i = \overline{g}_i.
   \end{equation}
   \end{proposition}
   
   \begin{proof}
   The first formula follows from a straightforward  computation using the definition of $h'', h'$, Klee's equations, and the fact that the reduced Euler characteristic is the alternating sum of the reduced Betti numbers.  The second formula  is \cite[Corollary 5.6]{MN} and is equivalent to the first via Poincar\'e duality  
   \end{proof}

\begin{theorem} \label{MN gtilde} \cite[Theorem 5.4]{MN}
If $\Delta$ is an $\field$-homology manifold and the links of the vertices of $\Delta$ have weak Lefschetz elements, then $\overline{g}_i \ge 0.$   Furthermore, if $|\Delta|$ is $\field$-orientable, then $(\hat{g}_0, \dots, \hat{g}_{\lfloor d/2 \rfloor})$ is an M-vector.
\end{theorem}

\begin{corollary}
If $|\Delta|$ is an orientable $\field$-homology  manifold and $d \ge 4,$ then $\hat{g}_3 \le \hat{g}^{<2>}_2.$
\end{corollary}

\begin{proof}
Murai and Nevo's proof of Theorem \ref{MN gtilde} without the assumption of weak Lefschetz elements, but with the knowledge that $\cdot \omega: \field(\Delta)_1 \to \field(\Delta)_2$ is injective, produces a ring whose Hilbert function is $(1, \hat{g}_1, \hat{g}_2, \hat{g}_3 + a, \dots)$ where $a \ge 0.$  
\end{proof}

The upper bounds for $\hat{g}_i,~i \ge3$ in the last Theorem and Corollary are sharp for arbitrary values of $\beta_1.$     Specifically, given an M-vector 
$(1, g_1, g_2, \dots, g_{\lfloor d/2 \rfloor})$ and $\beta \ge 0$, then there is a
manifold $\Delta$ so that $\beta_1(\Delta)=\beta$ and $\hat{g}_i (\Delta) = g_i$ for $i \ge 2.$  Start with a Billera-Lee polytope boundary whose $g$-vector is $(1, g_1, g_2, \dots, g_{\lfloor d/2 \rfloor}).$ Now subdivide facets enough times so that it is possible to identify $\beta$ pairs of facets in an orientable fashion and still have a simplicial complex.  After removing the interiors of the identified facets you will be left with a manifold  with the required properties. 

Whether or not nonorientable $\Delta$ whose links have weak Lefschetz elements satisfy $\hat{g}_i \ge 0$ is not known yet.   For evidence in favor, see the $i$-stacked manifolds in Section \ref{balls}.
 
Under certain circumstances this circle of ideas can be applied to pseudomanifolds.   In a triangulation of a normal three-dimensional pseudomanifold the link of a vertex is a connected compact surface without boundary.    The link types which are not spheres are called singular  and via excision are easily seen to be topological invariants of the pseudomanifold.   The following bound on $g_3$ for normal three-dimensional pseudomanifolds allowed Akhmejanov \cite{Akh} and Novik-Swartz \cite{NovSw4} to determine the complete set of possible $f$-vectors for a number of examples. 

\begin{theorem} 
Suppose $\Delta$ is a normal three-dimensional pseudomanifold with $N$ singular vertices.  Then $g_3 \le {N-3 \choose 3}.$
\end{theorem}

\begin{proof}
Suppose $v$ is a nonsingular vertex of the triangulation.  Every two-sphere is a polytope, so the $g$-theorem \cite{St80} implies that for generic $\omega$ and l.s.o.p. $\Theta$ multiplication $\cdot \omega: \field(\lk v)_1 \to \field(\lk v)_2$ is a surjection.  By the cokernel lemma $\dim_\field \field(\Delta)_3 - \dim_\field \field(\Delta)_2 \le {N-3 \choose 3}.$  The theorem now follows from Theorem \ref{rigidity} (\ref{rigidity h2}) and (\ref{rigidity h3}). 
\end{proof}

\section{PL-spheres} \label{PL-spheres}

One of the most intriguing classes of spheres are  PL-spheres.  One reason reason for this is that when studying triangulations of smooth compact manifolds these are precisely the links of the vertices that occur \cite{whitehead}.  There are several approaches one can take toward defining PL-spheres.  One is to define   a PL-sphere as a simplicial complex $\Delta$ which has a common stellar subdivision with the boundary of the simplex.  In that direction Babson and Nevo proved that if $\Delta'$ is obtained from $\Delta$ by a stellar subdivision on a face $\sigma$ and both $\R[\Delta]$ and $\R[\lk \sigma]$ have Lefschetz elements, then $\R[\Delta]$ also has Lefschetz elements \cite{BN}.  If one could prove that inverse stellar subdivisions also have this property then we would know that face rings of PL-spheres have Lefschetz elements.  As an alternative we will take Pachner's characterization of PL-spheres via bistellar moves as our definition.  

Let $A$ and $B$ be disjoint subsets of the vertices of $\Delta.$  Assume that $|A| + |B| = d+1$  and that the vertex induced subcomplex of $\Delta$ on $A \cup B$ is $A \star \partial B,$ the join of the simplex whose vertices are $A$ and the boundary of the simplex whose vertices are $B.$  The {\bf join} of two vertex disjoint simplicial complexes $\Delta'$ and $\Delta''$ is denoted $\Delta' \star \Delta''$ and is equal to $\{\sigma \cup \tau: \sigma \in \Delta', ~\tau \in \Delta''.\}$  A $(|B|-1)$-bistellar move consists of removing $A \star \partial B$ and replacing it with $\partial A \star B.$  A PL-sphere is any complex that can be obtained from the boundary of a simplex by a sequence of bistellar moves \cite{Pac}.  Since $\field[\partial \Delta^d]$ has Lefschetz elements, one obvious approach to the $g$-conjecture for PL-spheres is to show that the existence of weak Lefschetz elements  is preserved by bistellar moves.  The main result of this section is that if this is not true, then it is false in an `interesting' way.  See the paragraph following Theorem \ref{odd PL}.

\begin{theorem} \label{even PL}
Let $\Delta$ be a $2m$-dimensional $\field$-homology sphere and suppose that  $\Delta^\prime$ is obtained from $\Delta$ via a bistellar move with $|A| \neq |B|.$ Then $\field[\Delta^\prime]$  has weak Lefschetz elements if and only if $\field[\Delta]$ has weak Lefschetz elements.  
\end{theorem}

\begin{proof}
Throughout we assume that $\Theta$ is the ideal of $R$ generated by a generic l.s.o.p. and $\omega$ a generic one-form. 
Let $D$ be the closure of the complement of $A \star \partial B.$  A Mayer-Vietoris argument shows that $D$ is an $\field$-homology ball.  Consider the short exact sequence
\begin{equation} \label{PL ses}
0 \to I \to \field[\Delta] \to \field[A \star \partial B] \to 0.
\end{equation}
As an $R$-module the ideal $I$ is isomorphic to the ideal of $\field[D]$ generated by the interior faces of $D.$ By a theorem of Hochster \cite[II. Theorem 7.3]{St96} $I$ is isomorphic to the canonical module of $\field[D].$  Therefore, $\dim (I/ \Theta I) = h_{d-i}(D).$ If we quotient out by $\Theta$ in (\ref{PL ses}) we obtain the short exact sequence
\begin{equation} \label{PL lsop}
0 \to I/(I \cap \Theta) \to \field(\Delta) \to \field(A \star \partial B) \to 0.
\end{equation}
Thus $\dim I/(I \cap \Theta)_i = h_i(\Delta) -  h_i(A \star \partial B).$ Now,  $\Delta = D \cup (A \star \partial B)$ and $ D \cap (A \star \partial B) = \partial D.$ So $h_i(\Delta) = h_i(D) + h_i(A \star \partial B) - g_i(\partial D).$ Since $D$ is a homology ball, $g_i(\partial D) = h_i(D) - h_{d-i}(D)$ \cite{Grabe} and we conclude that $\dim I/(I \cap \Theta)_i = h_{d-i}(D).$ As there is a natural surjection from $I/ \Theta I \to I/(I \cap \Theta)$ we see that $I/(I \cap \Theta) \simeq I/\Theta I.$  

Now we repeat the argument for $\Delta^\prime,$
\begin{equation} \label{PL ses'}
0 \to I^\prime \to \field[\Delta^\prime] \to \field[\partial A \star  B] \to 0.
\end{equation}
The same reasoning as above show that $I^\prime$ is also isomorphic to the canonical module of $D$ and hence $I^\prime/ (\Theta \cap I^\prime) \simeq I/ (\Theta \cap I).$  In particular, if the  multiplication map
\begin{equation}  \label{lhs}
\cdot \omega: (I/ \Theta \cap I)_m \to (I/ \Theta \cap I)_{m+1}
\end{equation}
is an isomorphism, then
\begin{equation}  \label{lhs'}
\cdot \omega: (I^\prime/ \Theta \cap I^\prime)_m \to (I^\prime/ \Theta \cap I^\prime)_{m+1}
\end{equation}
is also an isomorphism.  

Look at the right-hand side of (\ref{PL lsop}) and the corresponding sequence for $\Delta^\prime.$  As long as $|A| \neq |B|$ the multiplication maps 
\begin{equation}  \label{rhs}
\cdot \omega: \field(A \star \partial B)_m \to \field(A \star \partial B)_{m+1}
\end{equation}
and
\begin{equation} \label{rhs'}
\cdot \omega: \field(\partial A \star B)_m \to \field(\partial A \star B)_{m+1}
\end{equation}
are both isomorphisms for generic choices of l.s.o.p and $\omega.$  Putting all this together, if generic choices of $\omega$ and $\Theta$ give isomorphisms $\cdot \omega: \field(\Delta)_m \to \field(\Delta)_{m+1}$, then (\ref{lhs}) and (\ref{rhs}) are isomorphisms.  So (\ref{lhs'}) and (\ref{rhs'}) are isomorphisms and  $\cdot \omega: \field(\Delta^\prime)_m \to \field(\Delta^\prime)_{m+1}$ is an isomorphism.
\end{proof}

The corresponding statement for odd-dimensional homology spheres is the following.  The proof is in the same spirit as  above with surjectivity between degrees $m$ and $m+1$ replacing isomorphism.  

\begin{theorem}  \label{odd PL}
Let $\Delta$ be a $2m-1$-dimensional $\field$-homology sphere and suppose $\Delta^\prime$ is obtained from $\Delta$ via a bistellar move with $|A| \neq m.$ Then $\field[\Delta^\prime]$  has weak Lefschetz elements if and only if $\field[\Delta]$ has weak Lefschetz elements.
\end{theorem}

One way to think about PL-triangulations of $S^{d-1}$ is through a graph whose vertices are the triangulations and edges represent bistellar moves.  Pachner proved that for fixed $d$ this graph is connected \cite{Pac}.  Suppose we color the vertices green for those triangulations whose face rings have weak Lefschetz elements and red otherwise.  The two previous theorems tell us that the colors can only change along edges with represent very specific bistellar moves.

Another consequence of Theorem \ref{even PL} is that it is sufficient to verify that for even-dimensional spheres the bistellar move with $|A|=|B|$ preserves the existence of weak Lefschetz elements in order to establish their existence for all PL-spheres.  The step to odd-dimensional PL-spheres can be accomplished using the cokernel lemma.

\section{Few vertices. Few edges.} \label{small g2}

In this section we demonstrate several techniques that will allow us to show that  if $\Delta$ has few vertices or few edges, then it satisfies all, or at least part, of the $g$-conjecture.  Throughout this section $d \ge 5.$  We start by considering simplicial spheres with few vertices.      Mani proved that any simplicial sphere with $d+3$ vertices is polytopal \cite{Man} and hence satisfies the $g$-theorem. Thus the first nontrivial case is when $\Delta$ is a  simplicial sphere with $d+4$ vertices.  This means that $g_2 \le 6,$ with $g_2=6$ indicating a two-neighborly triangulation.

\begin{theorem} \label{few vertices}
Suppose $\Delta$ is a simplicial sphere with $d+4$ vertices and $g_2 \le 5.$  Then for generic l.s.o.p. $\Theta$ and one-form $\omega$ multiplication $\cdot \omega: \C(\Delta)_{\lceil d/2 \rceil} \to \C(\Delta)_{\lceil d/2 \rceil +1}$ is a surjection.  Therefore the $g$-vector of $\Delta$ is an M-vector.    
\end{theorem}

\begin{proof}
Let us say that $\C(\Delta)$ has very weak Lefschetz elements if there exist $\Theta$ and $\omega$ which satisfy the conclusion of the theorem.  The only difference between very weak Lefschetz elements and weak Lefschetz elements is that if $d$ is odd, then multiplication from degree $(d-1)/2$ to $(d+1)/2$ may not be an isomorphism.  The same argument that weak Lefschetz elements guarantee that the $g$-vector is an M-vector works for very weak Lefschetz elements.  

Suppose $\Delta$ has $d+4$ vertices and does not have very weak Lefschetz elements.  Write $d=2m$ or $d=2m+1$ depending its parity.  
Since $\Delta$ has $d+4$ vertices any vertex which does not have all of the other vertices in its link is polytopal and therefore has weak Lefschetz elements.  Call any vertex whose link does not have weak Lefschetz elements bad.   The cokernel lemma implies that $\Delta$ must have at least $d+2$ bad vertices.  In fact, there must be at least $d+3$ bad vertices.  The proof of the cokernel lemma shows that if there are only $d+2$ bad vertices then the induced subcomplex on the bad vertices is $m+1$-neighborly.  When $d$ is odd the classical van Kampen-Flores theorem \cite{Flo}, \cite{Van} forbids the embedding of the $m$-skeleton of the $(2m+2)$-simplex in $\R^{2m}.$ When $d$ is even it is still impossible  to embed an $m+1$-neighborly complex with $d+2$ vertices in $\R^{2m-1}$ as its cone would embed in $\R^{2m}$ again violating  van Kampen-Flores.    Of course, if $d+3$ vertices  have every other vertex in their link, $\Delta$ must be two-neighborly and $g_2 = 6.$
\end{proof}

\begin{remark} \label{KS} Kalai and Sarkaria have another approach to proving PL-spheres have  weak Lefschetz elements based on the van Kampen-Flores obstruction.  They conjecture that if the van Kampen-Flores obstruction to embedding an $m$-dimensional complex $\Delta$ in $\R^{2m}$ is zero, then  $\C(\Delta)/(\omega)_{m+1}$ is zero \cite[Conjecture 27]{Ka02}. This would imply the existence of weak Lefschetz elements for even-dimensional PL-spheres and the cokernel lemma would prove their existence for odd-dimensional PL-spheres.  
\end{remark}

What about when $\Delta$ has few edges?   The fewest number of edges given a fixed number of vertices is governed by the rigidity inequality $g_2 \ge 0.$  The rest of this section is devoted to proving that if $\Delta$ is a rational homology sphere and $d \ge 6,$ then $(1,g_1,g_2,g_3)$ is an M-vector whenever  $g_2 \le 6.$  In particular this holds for $d+4$ vertices.  Along the way we will find that in some lower dimensions we can do a little better.   

In preparation for our journey we will need several preliminary results.  The first concerns facet connected sum.  Let $\Delta_1$ and $\Delta_2$ be two $(d-1)$-dimensional simplicial complexes with disjoint vertex sets.  Given facets   $\sigma_1$ of $\Delta_1$ and $\sigma_2$ of $\Delta_2,$ and a simiplicial isomorphism $\phi:\sigma_1 \to \sigma_2,$ the facet connected sum $\Delta_1 \#_\phi~\Delta_2$ is formed by identifying $\sigma_1$ and $\sigma_2$ using $\phi$ and removing the open face corresponding to the identified facets $\sigma_1$ and $\sigma_2.$  For example, a bistellar $0$-move is the same as taking a facet connected sum with the boundary of the $d$-simplex.  The boundary of the removed facet is a missing facet in $\Delta_1 \#_\phi~\Delta_2.$     A {\bf missing facet} in $\Delta$ is a subset $\sigma$ of vertices such that $|\sigma|=d, \partial \sigma \subseteq \Delta$, but $\sigma \notin \Delta.$  A rational homology sphere can be written as a facet connected sum if and only if it has a missing facet.  For a detailed recent account of how to write $\Delta = \Delta_1 \#_\phi~\Delta_2$ when $\Delta$ has a missing facet, see \cite[Section 3]{BD}.  Facet connected sum behaves quite well with respect to the $g$-vector and weak Lefschetz elements.  A straight-forward calculation shows that for $1 \le i \le d-1, ~g_i(\Delta_1 \# \Delta_2) = g_i(\Delta_1) + g_i(\Delta_2).$  

\begin{theorem}  \cite[Theorem 6.1]{BN}\label{connected sum}
Suppose $\Delta_1$ and $\Delta_2$ are $(d-1)$-dimensional rational homology spheres.  If $\field[\Delta_1]$ and $\field[\Delta_2]$ have Lefschetz (respectively weak Lefschetz) elements, then $\field[\Delta_1 \#~\Delta_2]$ also has  Lefschetz (respectively weak Lefschetz) elements.  
\end{theorem}

When $g_2$ equals zero $\Delta$ is a stacked sphere \cite{Ka87}.  A {\bf stacked sphere} is any complex that can be obtained by starting from the boundary of a simplex and performing a sequence of  bistellar-$0$ moves.  Repeated applications of the above theorem shows that $\field[\Delta]$ has Lefschetz elements if $\Delta$ is a stacked sphere.  

Another tool we require is edge contraction.   Let $u$ and $v$ be vertices of $\Delta$ which share an edge $e$.  We say $e$ satisfies the {\bf link condition} if $\lk e = \lk u \cap \lk v.$  When $e$ satisfies the link condition then it is possible to contract $e$ identifying $u$ and $v$ to obtain a new complex $\Delta_e$ which is PL-homeomorphic to $\Delta$.  Furthermore, if $\field[\Delta_e]$ and $\field[\lk e]$  have  Lefschetz elements, then so does $\field[\Delta].$  In characteristic zero this is due to Babson and Nevo \cite{BN}.  Murai extended their result to arbitrary characteristics in \cite{Mur}.

The next ingredient we will use in our study of $\Delta$ with small $g_2$ is Nevo and Novinsky's classification of rational homology spheres with $g_2=1.$  

\begin{theorem} \label{g2=1} \cite{NN}
Let $\Delta$ be a  rational homology sphere with $g_2=1$ and $d \ge 4.$ If $\Delta$ is not a facet connected sum with a stacked sphere, then $\Delta$ is either the join of the boundary of two simplices each of which has dimension at least two, or $\Delta$ is the join of the boundary of the $(d-2)$-simplex and the boundary of a polygon.  
\end{theorem}

Now let $\Delta$ be a four-dimensional homology sphere with $g_2=1.$  We wish to show that $\field[\Delta]$ has  Lefschetz elements.   By Theorem \ref{connected sum} and Theorem \ref{g2=1} we need only prove that $\field[\Delta]$ has  Lefschetz elements when $\Delta$ is the join of the boundary of a polygon and the boundary of the three-simplex.  Contracting the edges of the polygon until a triangle is reached reduces the problem to $\field[\partial \Delta^2 \star \partial \Delta^3]$ where the existence of weak Lefschetz elements is an easy exercise.  

The last new trick before we begin our journey involves the short simplicial $g$-vector.  This is the sum  of the $g$-vectors of the links of all of the vertices.    Specifically, define
$$\tilde{g}_i(\Delta) = \sum_{v \in [f_0]} g_i(\lk v).$$

\begin{proposition} \cite{Sw2}
  If $\Delta$ is a pure $(d-1)$-dimensional complex, then
  \begin{equation} \label{g-tilde}
  \tilde{g}_2 = 3 g_3 + (d-1) g_2.
  \end{equation}
\end{proposition}

 While our problem begins in dimension four at $g_2=2,$ we will begin by proving that if $\Delta$ is a four-dimensional rational homology sphere  with $g_2 \le 3,$ then $\field[\Delta]$ has weak Lefschetz elements.    In order to facilitate the proofs we introduce the notion of a minimal non-Lefschetz $(d-1)$-sphere.  We call $\Delta$ a {\bf non-Lefschetz} sphere if $\Delta$ is a rational homology sphere whose face ring does not contain weak Lefschetz elements.  A $(d-1)$-dimensional non-Lefschetz sphere $\Delta$ is a {\bf minimal non-Lefschetz sphere} if among all $(d-1)$-dimensional non-Lefschetz spheres $\Delta$ minimizes $g_2$ and in addition, $\Delta$ minimizes the number of vertices of  $(d-1)$-dimensional non-Lefschetz sphere with the same $g_2.$  Of course, if the most optimistic version of the $g$-conjecture holds, then there are no non-Lefschetz spheres!  

\begin{lemma}
If $\Delta$ is a minimal non-Lefschetz sphere, then $\Delta$ does not have a missing facet. 
\end{lemma}

\begin{proof}
Let $\Delta$ be a minimal non-Lefschetz sphere.  If $\Delta$ has a missing facet, then $\Delta = \Delta' \#~\Delta''$ with $g_2(\Delta') \le g_2(\Delta), g_2(\Delta'') \le g_2(\Delta)$ and both $\Delta'$ and $\Delta''$ have fewer vertices than $\Delta.$  By Theorem \ref{connected sum}, $\field[\Delta]$ has weak Lefschetz elements.
\end{proof}

\begin{lemma}
If $\Delta$ is a minimal non-Lefschetz sphere, then the link of every edge of $\Delta$ has at least $d$ vertices.  Equivalently, for every edge $e$ the link $\lk e \neq \partial \Delta^{d-2}.$
\end{lemma}

\begin{proof}
Suppose $e=\{u,v\}$ is an edge of $\Delta$ with  $\lk e = \partial \Delta^{d-2}= \partial \sigma,$ where $\sigma$ is a $d-1$ subset of the vertices.  We first observe that one of $\{u\} \star \sigma$ or $\{v\} \star \sigma$  is not a face of $\Delta.$  Otherwise $\Delta$ contains as a subcomplex the boundary of the $d$-simplex $\{u,v\} \cup \sigma.$  So assume $\{u\} \cup \sigma \notin \Delta.$  There are now two cases to consider.  One, $\sigma \in \Delta.$  In this case $\Delta$ contains the missing facet $\{u\} \cup \sigma.$  Two, $\sigma \notin \Delta.$  This implies that the induced subcomplex on the union of the vertices of $e$ and $\sigma$ is $e \star \partial \sigma.$  A $(d-2)$-bistellar move gives us a complex with smaller $g_2$ and, via the results of the previous section, weak Lefschetz elements for $\Delta.$  
\end{proof}

\begin{corollary} \label{no stacked spheres}
Suppose $\Delta$ is a minimal non-Lefschetz sphere and $v$ is a vertex of $\Delta.$  Then $\lk v$ is not a facet connected sum with a stacked sphere. In particular, $g_2(\lk v)>0.$
\end{corollary}

\begin{proof}
Suppose $\lk v = \Delta' \#~\Delta''$ with $\Delta''$ a stacked sphere.  Then $\lk v$ can be obtained from $\Delta'$ by a sequence $0$-bistellar moves.  The last one gives an edge whose link is $\partial \Delta^{d-2}.$  The only $\lk v$ with $g_2(\lk v)=0$ not forbidden by this argument is the boundary of the $(d-1)$-simplex.  But this link permits the use of a bistellar $(d-1)$-move.  
\end{proof}

\begin{lemma} \label{no link condition}
If $\Delta$ is a four-dimensional minimal non-Lefschetz sphere, then no edge of $\Delta$ satisfies the link condition.
\end{lemma}

\begin{proof}
If an edge $e$ does satisfy the link condition then we arrive at a contradiction by contracting $e.$  The homology sphere $\Delta'$ obtained by contracting $e$ has fewer vertices and $g_2(\Delta')$ is no bigger than $g_2(\Delta).$   Minimality says that $\field[\Delta']$ has weak Lefschetz elements and the the link of $e$ is two-dimensional and hence a polytope, so $\field[\Delta]$  has weak Lefschetz elements.  \end{proof}

We are finally prepared to begin the promised trek toward seeing that homology spheres with few edges satisfy the $g$-conjecture.  

\begin{theorem} \label{dim four}
If $\Delta$ is a four-dimensional minimal non-Lefschetz sphere, then $g_2(\Delta) \ge 4.$  Equivalently, if $\Delta$ is a four-dimensional rational homology sphere and  $g_2(\Delta) \le 3, $ then $\field[\Delta]$ has weak Lefschetz elements.
\end{theorem}

\begin{proof}
Throughout the proof we assume that $\Delta$ is a four-dimensional minimal non-Lefschetz sphere.  We will show that $g_2(\Delta)=0, 1, 2$ or $3$ is impossible.  The previous discussion eliminates $g_2$ equal to zero or one, so we can assume that $g_2$ is two or three.  We start with $g_2=2.$    From (\ref{g-tilde}), $\tilde{g}_2=8.$ If $\Delta$ has nine or more vertices, then $g_2(\lk v)=0$ for some vertex $v$ which violates Corollary \ref{no stacked spheres}. The minimum number of vertices required for $g_2=2$ is eight, so this is the only remaining case.   By Corollary \ref{no stacked spheres} we know that $g_2(\lk v)=1$ for every vertex of $\Delta.$   There is one edge missing from the complex, so there are six vertices whose link has seven edges.  Let $v$ be one of them.  By Lemma \ref{g2=1} the link of $v$ is the join of a three-cycle, say $x,y,z$ and a four-cycle, say $a,b,c,d$ in cyclic order.   The edge $e=\{v,a\}$ must violate the link condition.  Since the link of $e$ already contains all of the faces of the link of $v$ restricted to the vertices of the link of $e,$ the only way $e$  can violate the link condition is if the link of $a$ contains the vertex $c$ which is not in the link of $e,$ but is in the link of $v.$  Similarly, the edge $\{b,d\}$ must be in $\Delta.$  But we can now see that $\Delta$ is two-neighborly, which is impossible since $g_2=2.$  

From here on we assume that $g_2(\Delta)=3.$  So $\tilde{g}_2=12.$   If $\Delta$ has thirteen or more vertices, then the link of some vertex is a stacked sphere, which is impossible.  It remains to eliminate $8,9,10,11$ and $12$ as a possible number of vertices.  We begin with eight.  This means that the $h$-vector is $(1,3,6,6,3,1)$, so the $f$-vector is $(1,8, 28, 52, 50, 20).$  While $\Delta$ has all possible edges, it has four missing triangles.   What are they?  If a vertex appears in three of them, then its link will have $g_2=0$ which is forbidden.  If an edge appears in two of the missing triangles, then its link will have only $4$ vertices, also impossible.  Let us consider the tetrahedra in $\Delta.$  Suppose $\sigma$ is a subset of the vertices of cardinality four.  In order for $\sigma$ to be a face of $\Delta$ none of the missing triangles can be in $\sigma.$  Each missing triangle removes five four-sets from potential consideration and no four-set is eliminated by two distinct missing triangles since no two missing triangles share an edge.  Thus there are $70-20=50$ candidates.  By Alexander duality a $\sigma$ which does not contain a missing triangle is in $\Delta$ if and only if the vertex induced subcomplex on its complement has the homology of a point.  Since no four-subset of vertices is missing two triangles all $50$ candidate subsets are indeed tetrahedra of $\Delta.$   As there are in fact $50$ tetrahedra in $\Delta$ we try counting $4$-simplices.  

Each missing triangle eliminates $10$ potential five-subsets of vertices.  However, a five-subset is double counted whenever a pair of missing triangles share a vertex.  Let $m$ be the number of pairs of missing triangles which share a vertex.  Then there $56-40+m$ candidate five-subsets for the $4$-simplices of $\Delta.$ Alexander duality shows that all of these candidate five-subsets are $4$-simplices of $\Delta.$   There are twenty $4$-simplices, so $m$ is four. 

 The vertices of $\Delta$ are $\{1, \dots, 8\}.$  W.L.O.G. assume that $\{1,2,3\}$ and $\{6,7,8\}$ are two missing triangles with no common vertex.  At least two of these six vertices are in two missing triangles, so W.L.O.G assume $1$ is such a vertex.  Up to simplicial isomorphism there are two possibilities for a third triangle: $\{1,4,5\}$ or $\{1,4,6\}.$  If $\{1,4,5\}$ is the third triangle, then, up to simplicial isomorphism, the only possibility for a fourth triangle is $\{2,4,6\}.$  However, this is not possible since the link of $\{1,2,4\}$ is just $\{7,8\}.$
  
 This leaves $\{1,2,3\}, \{1,4,6\}$ and $\{6,7,8\}$ as three of the missing triangles.  Up to simplicial isomorphism the last missing triangle must be $\{2,5,7\}.$  However, with these four missing triangles the link of $\{1,2,5\}$ consists of the two edges $\{4,8\}$ and $\{6,8\}.$  Thus we conclude that if $g_2(\Delta)=3,$ then $\Delta$ does not have eight vertices.  
 
 Lest the reader despair, the proofs that the number of vertices can not be $9,10,11$ or $12$ are simpler as we can now count missing edges.  Let us examine the possibility that $\Delta$ has nine vertices.  
There must be at least six vertices whose links have $g_2 = 1.$  What can those links be?  According to Theorem  \ref{g2=1} they must all be of the form $T \star P$ where $T$ is a circuit of length $3$ and $P$ is a circuit of length $3,4$ or $5.$  Let $v$ be a vertex such that $g_2(\lk v) =1$ and $P$  is a circuit of length $5.$  As was the case when discussing $g_2=2,$ in order to guarantee  that every edge of $\Delta$ violates the link condition every vertex  in the $5$-circuit $P$ must be part of a chord of $P$ not in the link of $v.$    Thus there are at least three additional chords of $P$ in $\Delta.$   All of the vertices of $\Delta$ are in the closed star of $v$, $\Star v, $ and we have now seen that $\Delta$ must have all possible edges except for possibly two. Since $g_2(\Delta)$ is three, $\Delta$ is missing three edges!  So $P$ is not a circuit of cardinality five.  

Could $P$ have cardinality four?  The same argument as above shows that $P$ requires two chords.  Now the eight vertices in the closed star of $v$ have all possible edges between them.  So all of the missing edges are incident to the remaining vertex and its link is missing three vertices.   This implies that the link of this last vertex has only five vertices and  must be the boundary of the $4$-simplex.  As we have already seen that this is not allowed, we consider whether or not the link of $v$ and the other vertices whose links have $g_2=1$ could be the join of two triangles.  In this case each of these vertices would be missing two of the other eight vertices from their link.  Since there are at least six vertices with $g_2=1,$ the complex would be missing at least six edges.  This contradiction finishes the proof that $f_0(\Delta) \neq 9.$

The remaining cases of $10,11$ and $12$ vertices follow the same type of reasoning and  the details are left to the extremely interested reader.  
 Use the failure of the link condition in a vertex whose link has $g_2=1$ to add enough edges so that there is some other  link which is too small.  
\end{proof}

\begin{remark}
If $\Delta$ is a four-dimensional $\field$-homology sphere with $g_2 \le 3,$ then $\Delta$ is a PL-manifold.  In fact, any four-dimensional $\field$-homology manifold $\Delta$ with $g_2 < 17$ is actually a PL-manifold. To see this, let $v$ be a vertex of $\Delta$ whose link is not $S^3$.  This link must be a three-dimensional rational homology sphere.  All three-dimensional rational homology spheres are three-manifolds and any triangulation of $S^3$ is a PL three-sphere.  The main results of \cite{Wal} imply that $g_2(\field[\lk v]) \ge 17.$  Now consider the surjective map
$$ \field(\Delta)/ (\omega) \to \field(\Star v)/(\omega),$$
where $\omega$ and the l.s.o.p. are chosen to satisfy Theorem \ref{rigidity}.  Conversely, if the characteristic of $\field$ is not two, then there is a  non-PL $\field$-homology manifold $\Delta$ with $g_2(\Delta)=17.$  An appropriately chosen  one-vertex suspension of Walkup's minimal triangulation of $\R P^3$ has $g_2=17.$ See, for instance, \cite{JL} for a description of one-vertex suspensions.  If the characteristic of $\field$ is two, then it is not clear what the minimum $g_2$ is for the first non-PL $\field$-homology manifold.  A lower bound is  $g_2 = 21$ as there are no three-dimensional $\field$-homology spheres with $g_2 \le 20$ \cite{LSS}.  An upper bound is $g_2=28$ as there is a triangulation of the the $\Z_3$ lens space with $g_2=28.$  This triangulation is two-neighborly, so any one-vertex suspension will have $g_2=28.$  \end{remark}

\begin{theorem} \label{dim five}
If $\Delta$  is a five-dimensional non-Lefschetz sphere, then $g_2(\Delta) \ge 5.$  Equivalently, if $\Delta$ is a five-dimensional rational homology sphere and $g_2(\Delta) \le 4,$ then $\field[\Delta]$ has weak Lefschetz elements.
\end{theorem}

\begin{proof}
Let $\Delta$ be a five-dimensional non-Lefschetz sphere.    As outlined in the proof of Theorem \ref{few vertices}, $\Delta$ must have at least eight  vertices whose links are non-Lefschetz spheres (nine if $\Delta$ is homeomorphic to a sphere).   By the previous theorem the link of each of these vertices has $g_2 \ge 4$ and by Corollary \ref{no stacked spheres} $g_2$ of the links of the other vertices is at least one.  In order for the link of a vertex to have $g_2 \ge 4$ the link must have at least nine vertices.  So $\Delta$ has at least ten vertices. Thus $\tilde{g}_2$ is at least $34.$  If $g_2 \le 4$ then $g_3 \le 5.$ However, if $g_3=5,$ then Corollary \ref{upper bound} implies that multiplication by generic $\omega$ induces an injection from degree two to three, and hence a surjection from degree three to degree four.  Therefore    $g_3 \le 4.$   This leads to a contradiction as (\ref{g-tilde}) gives $\tilde{g}_2 \le 3 \cdot 4 + 5 \cdot 4 = 32.$
\end{proof}

\begin{theorem}
If $\Delta$ is a $6$-dimensional rational homology sphere with $g_2 \le 5,$ then $\Delta$ has very weak Lefschetz elements.  Furthermore, $\Delta$ is an integer homology sphere.  
\end{theorem}

\begin{proof}
We start by showing that $\Delta$ is an integral homology sphere.  Since $g_2 \le 5$, \cite[Theorem 5.2]{NovSw} shows that $H_1(\Delta; \field)=0$ for any $\field.$   In addition  $g_2 \le 5$   forces $H_2(\Delta; \field)=0=H_3(\Delta; \field)$ for any $\field$ \cite[Theorem 4.18]{Sw}.

We continue by way of contradiction and assume that  for generic $\omega$ multiplication $\cdot \omega: \field(\Delta)_4 \to \field(\Delta)_5$ is not surjective.  Let $B$ be the set of vertices of $\Delta$ whose links are non-Lefschetz spheres.  Since $d = 7$ the cardinality of $B$ is at least nine.   Can $|B|=9?$  Let $\Delta_B$ be the vertex induced subcomplex of $\Delta$ on B.  If $|B|=9,$ then $\Delta_B$ contains all possible  four-simplices for those nine vertices.  By the Canon-Edwards double suspension theorem \cite{Can}, the double suspension $\Sigma \Sigma \Delta$ is homeomorphic to the eight-sphere.  In addition,  $\Sigma \Sigma \Delta$ contains an embedding of the $4$-skeleton of the ten-dimensional simplex whose vertices consist of $B$ and a pair of connected  suspension points.  This contradicts  van Kampen-Flores.  

  The previous theorem tells us that any $\Delta$ with $|B| \ge 10$ has  $\tilde{g}_2 \ge 50.$  Since $g_2 \le 5$ we know that $g_3 \le 7$ and if $g_3 =7, $ then $\cdot \omega: \field(\Delta)_2 \to \field(\Delta)_3$ is injective,  and hence $\cdot \omega: \field(\Delta)_4 \to \field(\Delta)_5$ is surjective.  However, (\ref{g-tilde}) and the fact that $g_3 \le 6$ show that $\tilde{g}_2 \le 3 \cdot 6 + 6 \cdot 5 = 48.$  

\end{proof}

\begin{theorem}
If $\Delta$ is a $(d-1)$-dimensional rational homology sphere  with $d \ge 8$ and $g_2 \le 5,$ then for generic $\omega$ and $\Theta$ multiplication $\cdot \omega : \field(\Delta)_{d-3} \to \field(\Delta)_{d-2}$ is surjective. 
\end{theorem}

\begin{proof}
We proceed by induction on $d$ beginning with $d=7,$ the previous theorem.  Let $B$ be the subset of vertices of $\Delta$ for which  there are no one-forms $\omega$ such that  $\cdot \omega: \field(\lk v)_{d-4} \to \field(\lk v)_{d-5}$ is  surjective.   

For the induction step, we now know that in any minimal counterexample the link of any vertex in $B$  has $g_2 \ge 6$, the cardinality of $B$ is at least $d+2,$ and there is at least one other vertex whose link has positive $g_2.$    This gives a lower bound of $6(d+2) + 1 = 6d+13$ for $\tilde{g}_2.$  Since $g_2 \le 5$ means that $\tilde{g}_2 \le 3 \cdot 7 + (d-1) \cdot 5 = 5d+16$ we are done.
\end{proof}

\begin{theorem}
Let $\Delta$ be a $(d-1)$-dimensional rational homology sphere.  Then $(1,g_1,g_2,g_3)$ is an M-vector whenever $g_2$ is bounded by the following chart.  In particular, $g_3 \ge 0$ whenever $g_2 \ge 6.$  
$$\begin{array}{cc}
d & g_2 \le \\
6 & 7 \\
7 & 8 \\
d \ge 8 & 6 + \frac{22}{d-1}
\end{array}.
$$
\end{theorem}

\begin{proof}
As noted in Corollary \ref{upper bound} it suffices to show that $g_3 \ge 0$ under the stated conditions.  In order for $g_3$ to be negative, multiplication $\cdot \omega: \field(\Delta)_{d-3} \to \field_{d-2}$ by generic one-forms must not be surjective.   This will always require at least $d+4$ vertices.  Otherwise the links of the vertices have at most $d+3$ vertices and $g_2 \le 3.$   This would allow the application of the induction hypothesis and the cokernel lemma.  Now proceed exactly as in the previous three theorems except that $g_3$ is now bounded above by negative one.  
\end{proof}

\begin{remark}
There has been no attempt to get the best possible result of this nature.  It would be very surprising if a deeper investigation of the fact that in minimal non-Lefschetz spheres every edge fails the link condition and there are no missing facets did not produce substantially improved estimates. 
\end{remark}

\section{Even more optimistic} \label{k-cm}

   As noted earlier, the most optimistic form of the $g$-conjecture for sphere-like spaces includes rational homology spheres and eventually involves homology manifolds and pseudomanifolds.  If one leaves the world of manifolds or even pseudomanifolds, then there are even more optimistic possibilities.   
   
   \begin{definition}
   A $(d-1)$-dimensional complex $\Delta$ is {\bf Cohen-Macaulay} (CM) over $\field$ if for all faces (including the empty face) $\sigma,~\tilde{H}_i(\lk \sigma; \field) = 0 ~\forall  i < d-1-|\sigma|.$  We say $\Delta$ is $l$-CM if $\Delta$ is CM and removing $l-1$ or fewer vertices leaves a $(d-1)$-dimensional CM complex.  
   \end{definition}
   
   The term Cohen-Macaulay comes from Reisner's famous result that the face ring of $\Delta$ is Cohen-Macaulay if and only if $\Delta$ satisfies the above definition \cite{Rei}. The more specific $l$-CM complexes were introduced by Baclawski \cite{Bac}.  Both Cohen-Macaulay \cite{Mun} and $2$-CM \cite{Walker} are topological properties  of the complex, but $3$-CM and higher $l$-CM depend on the specific triangulation.  For instance, balls are $1$-CM but not $2$-CM, while spheres are $2$-CM.  The complete graph on $4$-vertices is $3$-CM, but if you subdivide any edge, then the resulting complex is only $2$-CM.  Examples of $l$-CM complexes with $l \ge 3$ include finite buildings \cite{Bjo}, independence complexes of matroids with no small cocircuits, and order complexes of geometric lattices with no short lines \cite{Bac}.

   Motivated by results on matroids \cite{HS}, \cite{Sw03}, and the more general spaces with convex ear decompositions \cite{Sw06}, Bj\"orner and Swartz proposed the following problem.
   
   \begin{problem} \cite{Sw06}
   Suppose $\Delta$ is a $2$-CM complex.  Do there exists one-forms $\omega$ such that multiplication $\cdot \omega^{d-2i}: \field(\Delta)_i \to \field(\Delta)_{d-i}$ is an injection for $0 \le i \le d/2?$
   \end{problem}  
   
   In combination with the Dehn-Sommerville equations (\ref{klee}) a positive solution to this problem would immediately imply the strongest form of the $g$-conjecture.    Nevo proved that the rigidity inequality holds for $2$-CM complexes.  
   
   \begin{theorem} \label{2-cm} \cite{Nev3}
   If $\Delta$ is $2$-CM and $d \ge 4,$ then for generic l.s.o.p. $\Theta$ and one-form $\omega$ multiplication $\cdot \omega: \field(\Delta)_1 \to \field(\Delta)_2$ is injective. 
   \end{theorem}
   
   One might be tempted to posit the existence of weak Lefschetz elements for $2$-CM complexes, but there are easy examples which show that the $h$-vector of a $2$-CM complex need not be unimodal.  For instance, let $\Delta_1$ be a three-sphere with $h$-vector $(1,10,40,10,1)$ and let $\Delta_2$ be the $3$-skeleton of the $11$-dimensional simplex.  So the $h$-vector of $\Delta_2$ is $(1,8,28,56,70).$  Now form $\Delta_3$ by identifying any two facets of $\Delta_1$ and $\Delta_2$.  The resulting complex is $2$-CM and has $h$-vector $(1,18, 68, 66, 71).$  Even if we restrict our attention to one of the most well-behaved classes of complexes, independence complexes of matroids without coloops, there might still be no weak Lefschetz elements.  
   
\begin{example} \label{nonfano}
Let $\overline{F}_7$ be the non-Fano matroid, the matroid obtained by relaxing one circuit-hyperplane of the Fano, and let $\overline{F}'_7$ be $\overline{F}_7$ with one free point added.  So $\overline{F}'_7$ is a rank three matroid on eight points and is representable over $\C.$  Now let $\Delta$ be the independence complex of $(\overline{F}'_7)^\star,$ the matroid dual of $\overline{F}'_7.$ Since $(\overline{F}'_7)^*$ is a rank five matroid without coloops, $\Delta$ is a four-dimensional $2$-CM complex.  (In fact, $\Delta$ is $3$-CM since the largest hyperplane of $(\overline{F}'_7)^*$ has cardinality five.)  By \cite{Huh} the $h$-vector of the independence complex of any matroid representable over $\C$ is log concave, hence unimodal.  Indeed, the $h$-vector of $\Delta$ is $(1,3,6,10,15,15).$  However, as the reader can check for themselves, for any l.s.o.p. $\Theta$ and one-form $\omega$ multiplication $\cdot \omega: \field(\Delta)_4 \to \field(\Delta)_5$ has nontrivial kernel.
\end{example}

While generalizing the existence of weak Lefschetz elements to $l$-CM complexes does not work, there is a different direction  that one might go in considering $g$-type results for $l$-CM complexes.  As far as we know the following question suggested by Reiner is completely open.

\begin{question} \cite{Reiner}
Let $\Delta$ be an $l$-CM complex.  Is it true that $h_0 \le h_1 \le \dots \le h_{\lfloor \frac{d\cdot(l-1)}{l}  \rfloor}?$
\end{question}

\section{Balls and other manifolds with boundary}  \label{balls}

In this section we will be deliberately vague in keeping track of what category of balls and spheres we are discussing.  We trust the reader to be sure that whatever constructions are under consideration are valid.  Certainly two likely candidates, PL-balls and PL-spheres, or $\field$-homology balls and $\field$-homology spheres would suffice.

It is hardly surprising that there are close connections between the study of $f$-vectors of spheres and $f$-vectors of balls.  The most obvious is that determining all possible $f$-vectors of balls includes determining  all possible $f$-vectors of spheres.  This can be seen in at least two different ways.  One approach is through coning.  Suppose  $(1, h_1, \dots, h_{d-1}, 1)$ is the $h$-vector of a sphere. Then $(1,h_1, \dots, h_{d-1},1,0)$ is the $h$-vector of the ball obtained by coning the sphere.  Conversely, suppose $(1, h_1, \dots, h_{d-1},1,0)$ is the $h$-vector of a ball such that $h_i= h_{d-i}$  for $0 \le i \le d.$  Now for any ball $\Delta,~g_i(\partial \Delta) = h_i(\Delta) - h_{d-i}(\Delta)$ \cite{Grabe}.  So the $h$-vector of the boundary sphere is $(1, h_1, \dots, h_{d-1},1).$  An alternative approach is to look at spheres with the interior of one facet removed.   If $(1,h_1, \dots, h_{d-1},1)$ is the $h$-vector of a sphere, then the $h$-vector of the ball obtained by removing the interior of a single facet is $(1,h_1, \dots, h_{d-1},0).$ On the other hand, if $(1, h_1, \dots, h_{d-1},0)$ with $h_i=h_{d-i}$ for $1 \le i \le d-1$ is the $h$-vector of a ball,  then the boundary of the ball is a simplex.  Filling in the boundary leaves a sphere with $h$-vector  $(1,h_1, \dots, h_{d-1},1).$

A less obvious connection  between the combinatorics of balls and spheres is that balls have Lefschetz and/or weak Lefschetz elements if and only if spheres do.  To make this statement precise we say a $(d-1)$-ball $\Delta$  has Lefschetz elements  if for generic choices of $\Theta$ and one-form $\omega,$ multiplication $\cdot \omega: \field(\Delta)_i \to \field(\Delta)_{d-i}$ is a surjection for $0 \le i \le d/2.$  Weak Lefschetz elements for balls are defined by insisting that multiplication $\cdot \omega: \field(\Delta)_{\lfloor d/2 \rfloor} \to \field(\Delta)_{\lfloor d/2 \rfloor +1}$ is a surjection.  Now let $\Delta'$ be the sphere obtained by coning off the boundary of $\Delta$ and consider the commutative diagram
$$ \begin{array}{ccccccccc}
0& \to& I& \to &\field(\Delta')_{i+1} &\to& \field(\Delta)_{i+1} &\to& 0\\
 & &\cdot \omega \uparrow & &\cdot \omega \uparrow& & \cdot \omega \uparrow& & \\
 0& \to& I& \to &\field(\Delta')_i &\to& \field(\Delta)_i &\to& 0
 \end{array}
 $$
 As pointed out by Stanley \cite{St93}, if the middle up arrow is surjective, then so is the right-hand up arrow.  So if spheres have Lefschetz elements (resp. weak Lefschetz elements), then balls do to.  Conversely, if balls have Lefschetz elements (resp. weak Lefschetz elements), then spheres do.  This is easily seen by again removing the interior of a facet from the sphere. 
 
 A much less well-known connection between balls and spheres is that there is an analog of Lefschetz elements for balls that involves isomorphism rather than surjectivity.  Consider the short exact sequence
 
 $$0 \to I \to \field[\Delta]/\Theta \to \field[\partial \Delta]/\Theta \to 0.$$
 
If $\omega$ is a  Lefschetz element for $\partial \Delta$ , then for $i \le d/2, ~\dim_\field (\field[\partial \Delta]/\Theta)_i = g_i(\Delta)$ and for $i > d/2, ~\field[\partial \Delta]/\Theta=0.$   Thus, for $i \le d/2$ the dimension of $I_i$ is $h_i(\Delta) - g_i(\partial \Delta) = h_{d-i}(\Delta)$  and the dimension $I_{d-i}=h_{d-i}(\Delta).$  A natural extension of the $g$-conjecture for spheres would be that multiplication by $\omega^{d-2i}: I_i \to I_{d-i}$ is an isomorphism.  If true this would imply the $g$-conjecture for spheres  by considering spheres with the interior of a single facet removed.

Another way in which the face rings of balls and spheres may interact is when the sphere is the boundary of a special ball.  We consider two examples.  The first uses properties of the face ring of an ambient sphere to prove the non negativity of the $g$-vector and was originally due to Kalai where he used algebraic shifting \cite{Ka91}.  Shortly thereafter Stanley gave an alternative proof \cite{St93}.   

\begin{theorem} \cite{Ka91}, \cite{St93}
Let $\Delta''$ be a $d$-homology sphere such that $\field[\Delta'']$ has Lefschetz elements.  Suppose $\Delta'$ is a $d$-homology ball which is a $d$-subcomplex of $\Delta''$ and set $\Delta = \partial \Delta'.$    Then for $i \le d/2,~ g_i(\Delta) \ge 0.$ 
\end{theorem}

Whether or not $g$-vectors of spheres $\Delta$ which satisfy the above hypotheses are M-vectors remains an open problem.  The largest test case is Murai's proof that Kalai's squeezed spheres have weak Lefschetz elements in characteristic zero \cite{Mur3}.   It is also very unclear what the class of spheres $\Delta$ which satisfy the above hypothesis look like.

  The second class of examples we consider are $i$-stacked spheres.  An $i$-{\bf stacked ball} is a ball $\Delta$ which has no interior faces of dimension $d-i-2$ or smaller.  For example,  a $2$-stacked $3$-ball has no interior  vertices and a stacked sphere is the same as a $1$-stacked sphere.  A sphere is $i$-stacked if and only if it is the boundary of an $i$-stacked ball.  Stacked spheres have a long history with Murai and Nevo's beautiful resolution of  McMullen and Walkup's greater lower bound conjecture being the most recent advance \cite{MN2}.    In the literature $i$ is usually restricted by $i \le d/2.$  As we will see in a moment, even for larger $i$ there are potentially interesting things to say about $i$-stacked spheres.

A simple relationship between stacked spheres and the face ring of their boundaries is easily seen using this very  general observation.

\begin{proposition}
Suppose $\Delta'$ is $d$-dimensional Cohen-Macaulay complex with $h_i(\Delta')=0$ and $\Delta$ is a $(d-1)$-dimensional subcomplex of $\Delta.$   Then for generic $\Theta$ and $\omega$ multiplication $\cdot \omega: \field(\Delta)_{i-1} \to \field(\Delta)_i$ is a surjection.
\end{proposition}

\begin{proof}
Since $\Delta$ is a subcomplex of $\Delta'$ there is a natural surjection $\field[\Delta']/(\Theta, \omega) \to \field[\Delta]/(\Theta,\omega).$ As $h_i(\Delta')=0,~\field[\Delta']/(\Theta, \omega)_i=0.$  Therefore $\field[\Delta]/(\Theta,\omega)_i$ is also zero.   
\end{proof}

\begin{corollary} \label{weak lefschetz for stacked spheres}
If $\Delta$ is an $i$-stacked sphere, then for generic $\Theta$ and $\omega$ multiplication $\cdot \omega: \field(\Delta)_i \to \field(\Delta)_{i+1}$ is surjective.  In particular, if $i \le d/2,$ then $\Delta$ has weak Lefschetz elements, and if $i > d/2, $ then $(1, g_1, \dots, g_{d-i})$ is an M-vector.
\end{corollary}

\begin{proof}
If $\Delta'$ is an $i$-stacked ball then it has no interior $(d-i-2)$-faces.  Hence $h_{i+1}(\Delta')=0.$  See, for instance,  \cite[Proposition 2.4]{McM04}.  
\end{proof}

\noindent The  above development is alluded to just before  \cite[Theorem 8]{St80}.  A consequence of the above proofs is that if $\Delta$ is an $i$-stacked sphere with $i \le d/2,$ then $\omega$ is a weak Lefschetz element for $\field[\Delta]$ under one of the most natural genericity conditions on $\Theta$ and $\omega.$  To be precise write each $\theta_i = a_{i,1} x_1 + \dots + a_{i,f_0} x_{f_0}$ and $\omega=a_{d+1, 1} x_1 + \dots + a_{d+1,f_0} x_{f_0}.$  A natural genericity condition on $\Theta$ and $\omega$ would be to insist that each $(d+1) \times (d+1)$ minor of the matrix $A=(a_{i,j})$ be nonsingular.  Since this requirement would insure that $\Theta \cup \omega$ is a l.s.o.p. for $\Delta'$, it also implies that $\omega$ is a weak Lefschetz element for $\field(\Delta).$  This is not true for all spheres whose face rings have weak Lefschetz elements.  Indeed, it is not true for smallest  sphere which is not $i$-stacked for $i \le d/2$, the boundary of the octahedron.  We have no idea, whether or not the class of spheres for which this natural genericity condition is sufficient to guarantee that $\omega$ is a weak Lefschetz element for $\field[\Delta]$ is interesting .

Recently Bagchi and Datta \cite{BD2}, and Murai and Nevo \cite{MN} introduced stacked manifolds, with and without boundary,  as generalizations of stacked balls and spheres.  An $\field$-homology manifold with boundary is $i$-stacked if it has no interior $d-i-2$ faces and an $\field$-homology manifold without boundary is $i$-stacked if it is the boundary of an $i$-stacked $\field$-homology manifold with boundary.  It turns out that for $i < d/2$  and $\partial \Delta= \emptyset$, $\Delta$ is $i$-stacked if and only if all of its links are \cite[Theorem 4.6]{MN}, \cite{BD2}.   The class of stacked manifolds has several interesting properties.

\begin{itemize}
\item
Stacked manifolds have appeared in several $f$-vector minimizing problems.  For instance, the minimal triangulations of the trivial $S^{d-2}$ bundle over $S^1$ ($d$-odd) and the nontrivial $S^{d-2}$ bundle over $S^1$ ($d$ even)  minimize the $f$-vector of $\field$-homology manifolds with nontrivial first Betti number. All of these complexes are $1$-stacked.  All of the known vertex minimal triangulations of $S^2 \times S^3$ are $2$-stacked \cite{MN}.  
\item
Stacked complexes figure prominently in the search for tight triangulations \cite{BD3}.  
\item
If $\Delta$ is an $\field$-homology manifold with boundary, then $\Delta$ is $i$-stacked if and only if $h''_{i+1}=0$ $(0 \le i \le d-1)$ \cite[Theorem 3.1]{MN}.  Interestingly, whether or not $\Delta$ is $i$-stacked can be determined without knowing the Betti numbers used to compute $h''.$  Let $\sigma$ be a nonempty face of $\Delta.$   Then $h_{d-|\sigma|}(\lk \sigma)$ is one if $\sigma$ is an interior face and zero if it is on the boundary.  So the number of $j$-faces in the interior of $\Delta$ is just the sum over all $j$-faces $\sigma$ of $h_{d-|\sigma|}(\lk \sigma).$  This sum can be computed directly from the $h$-vector of $\Delta$ via \cite[Proposition 4.10]{Sw06}.
\item
If $\Delta$ is an $i$-stacked PL-manifold with boundary, then $\Delta$ has a handle decomposition consisting of handles of index $i$ or less.  Let $\Delta^{\prime\prime}$ be the second barycentric subdivision of $\Delta.$  The closed stars of the barycenters of the interior faces of $\Delta$ in $\Delta^{\prime\prime}$ taken in any order in which the dimension of the faces is not increasing produces the desired handle decomposition   \cite[Chapter 6]{RS}.
\item
Which PL-manifolds with boundary which have handle decompositions using handles of index $i$ or less have $i$-stacked triangulations? For $i=0,1$ or $d-2$, the answer is all.  
\item
The last two comments apply to stacked PL-manifolds without boundary where the corresponding  surgery presentation replaces handle decomposition.  
\item
If $\Delta$ is an $i$-stacked $\field$-homology manifold without boundary and  $i < d/2,$ then $\field[\lk v]$ has weak Lefschetz elements for all vertices and 
Theorem \ref{MN gtilde} applies.

\item  
\cite[Proposition 5.2]{MN} shows that when  $\Delta$ is an $i$-stacked $\field$-homology manifold without boundary and $i < d/2$, then  $(1, \hat{g}_1, \dots, \hat{g}_{\lfloor d/2 \rfloor})$ is an M-vector even if $\Delta$ is not orientable.  
\item
If $\Delta$ is an $\field$-homology manifold without boundary, $i$-stacked, and $i < d/2,$ then $\hat{g}_{i+1} = 0.$  Conversely,  if $\Delta$ is an $\field$-homology manifold without boundary and  $\field[\lk v]$ has weak Lefschetz elements for every vertex $v$, then  $\Delta$ is $i$-stacked whenever $\hat{g}_{i+1}=0$ and $i+1<d/2.$ \cite{MN}.  
\item
If $\Delta$ is an $i$-stacked $\field$-homology manifold without boundary, then $\beta_j = 0$ for $ i+1 \le j \le d-i-2$ \cite{MN}.  This and the corresponding fact that $\beta_j=0$ for $j \ge i+1$ \cite{MN} when $\Delta$ is an $i$-stacked $\field$-homology manifold with boundary can also be deduced from the surgery/handle decomposition structure of $\Delta.$

\end{itemize}

For arbitrary manifolds with boundary the analogs of $h''$ and $\hat{g}$ are not yet well-established.  One possibility for $h''$ is  \cite[Theorem 3.2]{NovSw3}.  For $\hat{g}_i,$ with $i \ge 3$ there are no current candidates.  For $g_2$ there is the very promising suggestion of Kalai, $\gamma(\Delta) = h_2 - \mbox{ \# of interior vertices}.$  Kalai originally showed that when $d \ge 4$, $\gamma(\Delta) \ge 0$ for any normal pseudomanifold with nonempty boundary \cite{Ka87}.  Since $h_{d-2}(\lk v)$ is one when $v$ is an interior vertex and zero when it is a boundary vertex, \cite[Proposition 2.3]{Sw04} shows that $\gamma(\Delta) =h_2 - h_{d-1} - d h_d$ when $\Delta$ is an $\field$-homology manifold with boundary.

\begin{theorem}  \cite[Theorem 5.1]{NovSw3}
If $\Delta$ is an $\field$-homology manifold with orientable boundary and $d \ge 5,$ then
$$\gamma(\Delta) \ge {d \choose 2} \beta_1(\partial \Delta) + d~ \beta_0(\partial \Delta).$$
If $d=4,$ then
$$\gamma(\Delta) \ge 3 \beta_1(\partial \Delta) + 4 \beta_0(\partial \Delta).$$

\end{theorem}

As the $g$-conjecture for spheres in dimensions five and above demonstrates, a full accounting of the combinatorics of all of the triangulations of a given manifold has proven very difficult. There are only a few examples in dimension three \cite{Wal}, \cite{LSS} and four \cite{Sw}.  For manifolds with boundary this has proven even more so.  To get a sense of the difficulties we return to the case of balls.   Like spheres, there is a complete characterization of the $f$-vectors of balls   in dimensions four and below \cite{LS}, \cite{Kol}.  However, in higher dimensions the problem for balls may be significantly more difficult than for spheres.  As Kolins showed in \cite{Kol}  there is, at present, not even a credible {\bf conjectured} characterization of $f$-vectors of balls in dimensions six and above!

\noindent{\em 2010 Mathematics Subject Classification:}  05E45, 13F55.

\medskip \noindent{\em Keywords:} $g$-conjectures, $f$-vectors, face ring, Stanley-Reisner ring, Cohen-Macaulay, sphere


\begin{thebibliography}{999}

\bibitem{Akh}
T.~Akhmejanov, Trianglulations of normal $3$-pseudomanifolds on $9$ vertices, 2012, URL {\bf http: www.math.cornell.edu/$\sim$takhmejanov/pseudoManifolds.html}.

\bibitem{BN} E.~Babson and E.~Nevo, Lefschetz properties and basic constructions on simplicial spheres, J. Alg. Comb. 31 (2010), 111--129. 

\bibitem{Bac} K.~Baclawski, Cohen-Macaulay connectivity and geometric lattices, European J. Comb. 3 (1982), 293--305.

\bibitem{Bag}B.~Bagchi, The mu-vector, Morse inequalities and a generalized lower bound theorem for locally tame combinatorial manifolds, arXive:1405.5675 [math.GT].

\bibitem{BD} B.~Bagchi and B.~Datta, Lower bound theorem for normal pseudomanifolds, Expos. Math. 26 (2008), 327--351.

\bibitem{BD2} B.~Bagchi and B.~Datta, On $k$-stellated and $k$-stacked spheres, Disc. Math. 313 (2013), 2318--2329.

\bibitem{BD3} B.~Bagchi and B.~Datta, On stellated spheres and a tightness criterion for combinatorial manifolds, European J. Comb. 36 (2014), 294--313.

\bibitem{BilleraLee} L.~Billera and C.~Lee, 
A proof of the sufficiency of McMullen's conditions for $f$-vectors 
of simplical convex polytopes, 
J.~Comb.~Theory Series A 31 (1981), 237--255.

\bibitem{Bjo}
A.~Bj\"orner, Some combinatorial and algebraic properties of Coxeter complexes and Tits buildings, Adv. in Math. 52 (1984), 173--212.

 \bibitem{BjKa}
A.~Bj\"orner and G.~Kalai, An extended Euler-Poincar\'e theorem,
  Acta Math.~161 (1988),  279--303.
  
  \bibitem{BW} F.~Brenti and V.~Welker, $f$-vectors of barycentric subdivsions, Math. Zeit. 259(4) (2008),  849--865.
  
  \bibitem{Can}
  J.~Canon, Shrinking cell-like decompositions of manifolds in codimension three, Ann. of Math. (2) 110 (1979), 83--112.

\bibitem{Eis}
D.~Eisenbud, {\em Commutative algebra with a view toward algebraic
geometry}, Graduate Texts in Mathematics, 150, Springer-Verlag, New York,
1995.

\bibitem{Flo}
A.~Flores, \"Uber die Existenz $n$-dimensionaler Komplexes, die nicht in den $\R^{2n}$ topologisch einbettbar sind, Ergeb. Math Kolloqu. 5 (1933), 17--24. 


\bibitem{Grabe} H.~Gr\"abe, Generalized Dehn-Sommerville equations and an upper bound theorem, Beitr. Algebra Geom. 25 (1987), 47--60.

\bibitem{HS} T.~Hausel and B.~Sturmfels, Toric hyperK\"ahler varieties, Doc. Math. 7 (2002), 495--534.

\bibitem{Huh} J.~Huh, $h$-vectors of matroids and logarithmic concavity, arXiv: 1201.2915 [math.CO], 2012.

\bibitem{JL} M.~Joswig and F.~Lutz, One-point suspensions and wreath products of polytopes and spheres, J. Comb. Theory Ser. A 110 (2005), 193--216.

\bibitem{Ka87}
G.~Kalai, Rigidity and the lower bound theorem I,
Invent.~Math.~88 (1987),  125--151.

\bibitem{Ka91}
G.~Kalai,  The diameter of graphs of convex polytopes and $f$-vector theory, in ``Applied geometry and discrete mathematics", DIMACS Ser. Disc. Math. Theoret. Comp. Sci.,  American Mathematical Society, Providence, RI, 1991, 387--411.  



\bibitem{Ka02}
G.~Kalai, Algebraic shifting, in {\em Computational commutative algebra
and combinatorics} (Osaka, 1999), 121--163, Math.~Soc.~Japan, Tokyo, 2002.

\bibitem{Kar}
K.~Karu, The $cd$-index of fans and posets, Compos. Math. 142 (2006),  701--718.

\bibitem{Klee}
V.~Klee, A combinatorial analogue of Poincar\'e's  duality theorem,
Canadian J.~Math.~16 (1964), 517--531.

\bibitem{Kol} 
S.~Kolins, $f$-vectors of triangulated balls, Disc. Comp. Geom. 46 (2011),  427--446.
\bibitem{KN}
M.~Kubitzke and E.~Nevo, The Lefschetz property for  barycentric subdivisions of shellable complexes, Trans. Amer. Math. Soc. 361 (2009), 6151--6163.

\bibitem{Lee} C.~W.~Lee, Generalized stress and motions, in ``Polytopes: 
abstract, convex and computational" (Scarborough, ON, 1993), 249--271, 
NATO Adv.~Sci.~Inst.~Ser.~C Math.~Phys.~Sci., 440, Kluwer Acad. Publ., 
Dordrecht, 1994.

\bibitem{LS}
C.~Lee, L.~Schmidt, On the number of faces of regular triangulations and shellable balls, Rocky Mount. J. Math. 41 (2011), 1939--1961.

\bibitem{LSS}
F.~Lutz, T.~Sulanke, E.~Swartz, $f$-vectors of $3$-manifolds, Elec. J. Comb. 16 (2009), R13.


\bibitem{Mac} F.~Macaulay, Some properties of enumeration in the theory of modular systems, Proc. London Math. Soc. 26, (1927), 531--555.

\bibitem{Man} P.~Mani, Spheres with few vertices, J. Comb. Theory Ser. A 13 (1972), 346--352.

\bibitem{McM71}  
P.~McMullen, The number of faces of simplicial polytopes, Israel J. Math.  9  (1971),  559--570.

\bibitem{McM93}
P.~McMullen,  On simple polytopes, Invent. Math. 113(2) (1993), 419--444.

\bibitem{McM96}
P.~McMullen,  Weights on polytopes, Disc. Comp. Geom.  15(4) (1996), 363--388.

\bibitem{McM04} 
P.~McMullen, Triangulations of simplicial polytopes, Beitr\"age zur Algebra and Geometrie 45 (2004), 37--46.  

\bibitem{MS} P.~McMullen and G.~Shephard, {\em Convex Polytopes and the Upper Bound Conjecture}, London Mathematical Society Lecture Notes Series 3, Cambridge University Press, London, 1971.

\bibitem{Mun} J.~Munkres, Topological results in combinatorics, Mich. Math. J. 31 (1984), 113--128.

\bibitem{Mur07}
S.~Murai, Algebraic shifting of cyclic polytopes and stacked polytopes,
Disc. Math.~307 (2007), 1707--1721.

\bibitem{Mur3} 
S.~Mur, Generic initial ideals and squeezed spheres, Adv.  Math. 214 (2007), 701--729.  


\bibitem{Mur}
S.~Murai, Algebraic shifting of strongly edge decomposable spheres, J. Comb. Theory Ser. A  117 (2010), 1--16.

\bibitem{MN2} 
S.~Murai and E.~Nevo,  On the generalized lower bound conjecture for polytopes and spheres, Acta. Math. 210 (2013), 185--202.

\bibitem{MN} S.~Murai and E.~Nevo, On $r$-stacked triangulated manifolds, J. Alg. Comb. 39 (2014), 374--388.

\bibitem{MY}
S.~Murai and K.~Yanagawa, Squarefree $P$-modules and the $cd$-index, Adv. Math. 265 (2014), 241--279.

\bibitem{Nevo}
E.~Nevo, Algebraic shifting and $f$-vector theory, 
Ph.D.~thesis, Hebrew University, 2007.

\bibitem{Nev2}
E.~Nevo, Higher minors and van Kampen's obstruction, Math. Scand. 101 (2007), 161--176.

\bibitem{Nev3}  Rigidity and the lower bound theorem for doubly Cohen-Macaulay complexes, Disc. Comp. Geom. 39 (2008), 411--418.

\bibitem{NN} E.~Nevo and E.~Novinsky, A characterization of simplicial polytopes with $g_2=1$, J. Comb. Theory Ser. A 118 (2011), 387--395.

\bibitem{N98}
I.~Novik, Upper bound theorems for homology manifolds, 
Israel J.~Math.~108 (1998), 45--82. 

\bibitem{NovSw} 
I.~Novik and E.~Swartz, Socles of Buchsbaum modules, complexes and posets,  Adv.~Math. 222 (2009), 2059--2084.

\bibitem{NovSw2}
I.~Novik and E.~Swartz, Gorenstein rings through face rings of manifolds, Composit. Math. 145 (2009), 993--1000.

\bibitem{NovSw3} 
I.~Novik and E.~Swartz, Applications of Klee's Dehn-Sommerville relations, Disc. Comp. Geom. 42 (2009), 261--276.  

\bibitem{NovSw4}
I.~Novik and E.~Swartz, Face numbers of  pseudomanifolds with isolated singularities, Math. Scan. 110 (2012), 198--222.

\bibitem{Pac}
U.~Pachner, P.L. homeomorphic manifolds are equivalent by elementary shellings, European J. Comb. 12 (1991), 129--145.

\bibitem{Reiner} V.~Reiner, personal communication, 2005.

\bibitem{Rei}G.~Reisner, Cohen-Macaulay quotients of polynomial rings, Adv. Math. 21 (1976), 30--49.


\bibitem{RS}
C.~Rourke and B.~Sanderson, {\em Introduction to piecewise-linear topology}, Berlin Heidelberg New York: Springer-Verlag, 1982.

\bibitem{Sch}
P.~Schenzel,
On the number of faces of simplicial complexes
and the purity of Frobenius,  Math.~Z.~178 (1981), 125--142.

\bibitem{St77} R.~Stanley, Cohen-Macaulay complexes, in {\em Higher Combinatorics} (M.~Aigner ed.), Reidel, Boston/Dodrecht, 1977, 51--62.

\bibitem{St80}
R.~Stanley, The number of faces of a simplicial convex polytope,
Adv.~in Math.~35 (1980),  236--238.

\bibitem{St93}
R.~Stanley, A monotonicity property of $h$-vectors and $h^*$-vectors, Europ. J. Comb. 14 (1993), 251--258.  

\bibitem{St96}
R.~Stanley, {\em Combinatorics and Commutative Algebra},
Boston Basel Berlin:  Birkh{\"{a}}user, 1996.

\bibitem{Sw03} $g$-elements of matroid complexes, J. Comb. Theory Ser. B 88 (2003), 369--375.

\bibitem{Sw04}
E.~Swartz, Lower bounds for $h$-vectors of $k$-CM, independence, and broken circuit complexes, SIAM J. Disc. Math. 18 (2004/05), 647--661.

\bibitem{Sw06} $g$-elements, finite buildings and higher Cohen-Macaulay connectivity,  J. Comb. Theory Ser. A 113 (2006), 1305--1320. 

\bibitem{Sw} 
E.~Swartz, Face enumeration - from spheres to manifolds,  
J.~Eur.~Math.~Soc.~11 (2009), 449--485.

\bibitem{Sw2}
E.~Swartz, Topological finiteness for edge vertex-enumeration, Adv. Math. 219 (2008), 1722--1728.

\bibitem{Van}
E.~van Kampen, Komplexe in euklidischen R\"aumen, Abh. Math. Sem. Univ. Hamburg 9 (1932), 72--78.

\bibitem{Walker} J.~Walker, Topology and combinatorics of ordered sets, Ph.D. thesis, M.I.T., 1981.

\bibitem{Wal}
D.~Walkup, The lower bound conjecture for $3$- and $4$-manifolds.  Acta Math.  125 (1970), 75--107.

\bibitem{whitehead}
J.~Whitehead, On $C^1$-complexes, Ann. Math. 41 (1940), 809--824.

\end{thebibliography}
\end{document}